\documentclass{amsart}
\usepackage{enumerate,amssymb,amsmath,graphics,epsfig}
\usepackage[colorlinks]{hyperref}
\usepackage{subfigure}
\usepackage{color}
\usepackage{epstopdf}

\def\wbox#1;#2;{\vbox{\hrule\hbox{\vrule height#1mm\kern#2mm\vrule
  height#1mm}\hrule}}

\newcommand\RE{\mathbb{R}}

\newcommand\Huo{H^1_0(\Omega)}
\newcommand\Ld{L^2(\Omega)}
\newcommand\Vh{V_h^{nc}}
\newcommand\Vc{V_h^c}
\newcommand\Vt{\tilde V}

\newcommand\gh{\nabla_h}
\renewcommand\P[2]{P_{#1,\dots,#2}}
\newcommand\Pch{P^c_h}
\newcommand\Thc{T^c_h}
\newcommand\Pc[2]{P^c_{#1,\dots,#2,h}}
\newcommand\Ph[2]{P_{#1,\dots,#2,h}}
\newcommand\E[2]{E_{#1,\dots,#2}}
\newcommand\Eh[2]{E_{#1,\dots,#2,h}}
\newcommand\Ech[2]{E^c_{#1,\dots,#2,h}}
\renewcommand\L{\mathcal{L}}
\newcommand\lch{\lambda^c_h}
\newcommand\uch{u^c_h}
\newcommand\lc{\lambda^c}
\newcommand\uc{u^c}
\newcommand\rayl{\mathcal{R}}
\newcommand\raylh{\mathcal{R}_h}
\newcommand\im{\mathrm{Im}}
\newcommand\spann{\mathop{\rm span}\nolimits}
\newcommand\nh{\nabla_h}

\newcommand\tv{\tilde v}

\let\NEW\relax

\theoremstyle{plain}
\newtheorem{thm}{Theorem}
\newtheorem{proposition}[thm]{Proposition}
\newtheorem{lemma}[thm]{Lemma}

\theoremstyle{remark}

\begin{document}
\title[]
{A posteriori error analysis for nonconforming approximation of multiple
eigenvalues}

\dedicatory{Dedicated to Prof.\ Martin Costabel on the occasion of his 65th
anniversary}

\author{Daniele Boffi}
\address{Dipartimento di Matematica ``F. Casorati'', Universit\`a di Pavia,
Italy}
\email{daniele.boffi@unipv.it}
\urladdr{http://www-dimat.unipv.it/boffi/}
\author{Ricardo G. Dur\'an}
\address{Departamento de Matem\'atica, Facultad de Ciencias Exactas y
Naturales, Universidad de Buenos Aires and IMAS, CONICET, 1428 Buenos Aires,
Argentina}
\email{rduran@dm.uba.ar}
\urladdr{http://mate.dm.uba.ar/~rduran/}
\author{Francesca Gardini}
\address{Dipartimento di Matematica ``F. Casorati'', Universit\`a di Pavia,
Italy}
\email{francesca.gardini@unipv.it}
\urladdr{http://www-dimat.unipv.it/gardini/}
\author{Lucia Gastaldi}
\address{DICATAM Sez. di Matematica, Universit\`a di Brescia, Italy}
\email{lucia.gastaldi@unibs.it}
\urladdr{http://www.ing.unibs.it/gastaldi/}

\begin{abstract}
In this paper we study an a posteriori error indicator introduced in E.~Dari,
R.G.~Dur\'an, C.~Padra, \emph{Appl.\ Numer.\ Math.}, 2012, for the
approximation of \NEW{the} Laplace eigenvalue problem with Crouzeix--Raviart
non-conforming finite elements. In particular, we show that the estimator is
robust also in presence of eigenvalues of multiplicity greater than one. Some
numerical examples confirm the theory and \NEW{illustrate} the convergence of an
adaptive algorithm when dealing with multiple eigenvalues.
\end{abstract}
\maketitle
\section{Introduction}

Although the a posteriori error analysis for eigenvalue problems arising from
partial differential equations is a mature field of research, some intriguing
questions remain open when discussing the convergence of an adaptive scheme
for the approximation of eigenvalues with multiplicity greater than one.

In this paper we consider the approximation of Laplace eigenvalue by standard
Crouzeix--Raviart finite elements (see~\cite{CR} and, for
instance,~\cite{acta}).
In~\cite{DDP} an a posteriori error indicator has been proposed for this
problem and its efficiency and reliability have been proved. The analysis
of~\cite{DDP} showed that the indicator is equivalent to the energy norm of
the error in the eigenfunctions (up to higher order terms) and that it
provides an upper bound for the error in the \emph{first} eigenvalue (up to higher order
terms). In this paper we are mainly interested in the case when an eigenvalue
may have multiplicity greater than one. This topic has been the object of
little research and only very recently people started investigating the issues
originating from the presence of multiple eigenvalues (\NEW{see, in
particular,~\cite{ovall1,SoGi,ovall2,zhou,Gallistl}}).

The presented results contain a theoretical part, included in Sections~3
and~4, and some numerical experiments reported in Section~5.

In Section~3 we study the error estimates for the eigenfunctions and,
recalling the results of~\cite{DDP}, we show that the results extend in a
natural way to the case of multiple eigenvalues.
In Section~4, using some special tools adapted from~\cite{knyazev}, we extend
the estimates for the eigenvalues to the general case of multiplicity $q\ge1$.
One of the main difficulties comes from the fact that, when using
non-conforming finite elements, one cannot deduce from the min-max lemma that
the discrete eigenvalues should be upper bounds of the
corresponding continuous ones. In our analysis we study separately the cases
when an eigenvalue is approximated by $q$ discrete eigenvalues \NEW{from}
above or \NEW{from}
below. Our analysis does not apply to the case when a continuous eigenvalue
corresponds to discrete eigenvalues which can approximate it simultaneously
from above \NEW{or} from below. It should however be noted that in most situations
Crouzeix-Raviart element provides lower bound: this has been proved
asymptotically for
singular eigenspaces (see~\cite{AD} and~\cite{DDP}).
\NEW{See also~\cite{cg} where
this property has been used for the construction of guaranteed lower bounds
for eigenvalue approximation.}
Known examples of
discrete eigenvalues which provide approximation from above are rare and
computed on very coarse meshes.

The numerical results shown in Section~5 confirm the theory and aim at
investigating the behavior of an adaptive procedure based on the studied
indicator in case of multiple eigenvalues. As expected, it turns out that a
correct procedure should take into account all discrete eigenfunctions
approximating the same eigenspace (see~\cite{SoGi}). One of the main issues
raised by this investigation is that in general it is not known a priori
(besides very particular situations like the one considered in our tests) the
multiplicity of an eigenvalue of the continuous problem and it is not obvious
to detect which discrete values correspond to it. This phenomenon requires
further investigation and will be the object of future study.

\section{Setting of the problem}
Let $\Omega\subset \RE^d$, $d=2,3$ be a polygonal or polyhedral Lipschitz
domain, we consider the Laplacian eigenproblem:
find $\lambda\in\RE$ and $u\in\Huo$ with $u\ne0$ such that
\begin{equation}
\label{eq:var}
a(u,v)=\lambda(u,v)\quad\forall v\in\Huo,
\end{equation}
where
\[
a(u,v)=\int_\Omega \nabla u\nabla v\,dx\qquad (u,v)=\int_\Omega uv\,dx.
\]
It is well known that the eigenvalues of the problem above form an increasing
sequence tending to infinity:
\begin{equation}
\label{eq:lambda}
0<\lambda_1\le\lambda_2\le\dots\le\lambda_i\le\cdots
\end{equation}
We denote by $u_i$ an eigenfunction associated to the eigenvalue 
$\lambda_i$; it is well known that the eigenfunctions can be chosen such that
the following properties are satisfied:
\begin{equation}
\label{eq:orto}
\aligned
&(u_i,u_i)=1\qquad&&(u_i,u_j)=0\quad\text{if }i\ne j\\
&a(u_i,u_i)=\lambda_i\qquad&&a(u_i,u_j)=0\quad\text{if }i\ne j.
\endaligned
\end{equation}
Let us introduce the Crouzeix--Raviart non conforming finite element space we
shall work with (see~\cite{CR}).
We consider a regular family of decompositions of $\Omega$
into closed triangles or tetrahedra. Let $h_K$ denote the diameter of the
element $K$ and $h=\max_{K\in\mathcal{T}}h_K$. The set of all faces $F$ of
elements in $\mathcal{T}_h$ is denoted by $\mathcal{F}_h$. 
For any internal face $F$ let $K$ and $K'$ be two elements such that $K\cap K'=F$,
we denote by $[v]_F$ the jump across $F$ for $v\in L^2(K\cup K')$. For a
face $F\subset\partial\Omega$ we set $[v]_F=v$. Then we define
\[
\Vh=\{v\in\Ld:\ v|_{K}\in\mathcal{P}_1(K)\ \forall K\in\mathcal{T}_h 
\text{ and }\int_F[v]_F=0\ \forall F\in\mathcal{F}_h\}.
\]
We introduce the following discrete bilinear form defined on $\Vh\times\Vh$
\[
a_h(u,v)=\sum_{K\in\mathcal{T}_h}\NEW{\int_K}\nabla u\nabla v\,dx
=\int_\Omega \gh u \gh v\,dx \qquad\forall u,v\in\Vh
\]
where
\[
\gh u|_{\NEW{K}}=\nabla(u|_{\NEW{K}}).
\]
Let us recall some standard notation. We set $\|\cdot\|^2_0=(\cdot,\cdot)$,
the $L^2$-norm, and
\begin{equation}
\label{eq:norms}
\aligned
&\|u\|_1^2=a(u,u)=\|\nabla u\|^2_0&&\forall u\in\Huo\\
&\|u\|^2_h=a_h(u,u)=\|\nabla_h u\|^2_0&&\forall u\in\Vh
\endaligned
\end{equation}
Notice that thanks to the Poincar\'e inequality and to its discrete version for
non conforming elements (see~\cite{GastaldiNochetto})
both $\|\cdot\|_1$ and $\|\cdot\|_h$ are norms on $\Huo$ and $\Vh$, respectively.

Let $\Vt=\Huo+\Vh$, that is any element $\tilde u$ of $\Vt$ can be written
as the sum $\tilde u=u+u_h$ with $u\in\Huo$ and $u_h\in\Vh$.
We have that $\|\cdot\|_h$ is a norm in $\Vt$ and that in the case of $u\in\Huo$
it holds $\|u\|_h=\|u\|_1$.

Then the discrete eigenproblem reads:
find $\lambda_h\in\RE$ and $u_h\in \Vh$ with $u_h\ne0$ such that
\begin{equation}
\label{eq:pbnc}
a_h(u_h,v)=\lambda_h(u_h,v)\qquad\forall v\in \Vh.
\end{equation}
Problem~\eqref{eq:pbnc} admits exactly $N_h=\dim(\Vh)$ positive eigenvalues
with
\begin{equation}
\label{eq:lambdah}
0<\lambda_{1,h}\le\lambda_{2,h}\le\dots\le\lambda_{N_h,h}.
\end{equation}
Moreover, we denote by $u_{i,h}$ a discrete eigenfunction associated to the
eigenvalue  $\lambda_{i,h}$ with the following properties:
\begin{equation}
\label{eq:ortoh}
\aligned
&(u_{i,h},u_{i,h})=1\qquad&& (u_{i,h},u_{j,h})=0\quad\text{if }i\ne j\\
&a_h(u_{i,h},u_{i,h})=\lambda_{i,h}
\qquad && a_h(u_{i,h},u_{j,h})=0\quad\text{if }i\ne j.
\endaligned
\end{equation}
We indicate with $\E{i}{j}\subset\Huo$ (resp. $\Eh{i}{j}\subset\Vh$) the span
of the eigenvectors $\{u_i,\dots,u_j\}$ (resp. $\{u_{i,h},\dots,u_{j,h}\}$)
and $\P{i}{j}$ (resp. $\Ph{i}{j}$) the elliptic projection onto $\E{i}{j}$
(resp. $\Eh{i}{j}$), that is
\begin{equation}
\label{eq:proj}
\aligned
&\text{for }u\in\Huo,\ \P{i}{j}u\in\E{i}{j}\ \text{ s.t. }
a(u-\P{i}{j}u,v)=0&&\forall v\in\E{i}{j}\\
&\text{for }u\in\Vt,\ \Ph{i}{j}u\in\Eh{i}{j}\ \text{ s.t. }
a_h(u-\Ph{i}{j}u,v)=0&&\forall v\in\Eh{i}{j}.
\endaligned
\end{equation}
\NEW{The discrete solution operator $T_h:L^2(\Omega)\to L^2(\Omega)$ is
defined as $T_hf\in\Vh$ with
\begin{equation}
a_h(T_hf,v)=(f,v)\quad\forall v\in\Vh.
\label{eq:soluz}
\end{equation}
}%
In our a posteriori error analysis we shall also make use of the
space of conforming piecewise linear elements
\[
\Vc=\{v\in\Huo: v|_{K}\in\mathcal{P}_1(K)\ \forall K\in\mathcal{T}_h\}.
\]
The conforming discretization of the eigenvalue problem under consideration
reads: find $\lch\in\RE$ and $\uch\in\Vc$ with $\uch\ne0$
such that
\begin{equation}
\label{eq:pbconf}
a(\uch,v)=\lch(\uch,v)\qquad v\in\Vc.
\end{equation}
Problem~\eqref{eq:pbconf} admits $N^c_h=\dim(\Vc)$ positive eigenvalues
\begin{equation}
\label{eq:lambdac}
0<\lc_{1,h}\le\lc_{2,h}\le\dots\le\lc_{N^c_h,h}.
\end{equation}
As in the case of non conforming discretization we denote by $\uc_{i,h}$ the
eigenfunction associated to the eigenvalue $\lc_{i,h}$ such that
$(\uc_{i,h},\uc_{i,h})=1$ with the following orthogonality properties:
\begin{equation}
\label{eq:ortoch}
\aligned
&(\uc_{i,h},\uc_{i,h})=1\qquad&& (\uc_{i,h},\uc_{j,h})=0\quad\text{if }i\ne j\\
&a_h(\uc_{i,h},\uc_{i,h})=\lc_{i,h}
\qquad && a_h(\uc_{i,h},\uc_{j,h})=0\quad\text{if }i\ne j.
\endaligned
\end{equation}
Notice that $\Vc=\Huo\cap\Vh$, hence $N^c_h<N_h$ and
$\lambda_{i.h}\le\lc_{i,h}$ for $i=1,\dots,N^c_h$ because of the min max
characterization.

Let $\Pch$ be the elliptic projection from $\Vt$ onto $\Vc$, that is:
for all $u\in\Vt$, $\Pch u\in\Vc$ such that
\begin{equation}
\label{eq:defPch}
a(\Pch u,v)=a_h(u,v)\quad \forall v\in\Vc.
\end{equation}
Similarly to the nonconforming approximation, we denote by
$\Ech{i}{j}\subset\Vc$ 
the span of the eigenvectors $\{\uc_{i,h},\dots,\uc_{j,h}\}$ 
and by $\Pc{i}{j}$ the elliptic projection onto $\Ech{i}{j}$, that is:
for all $u\in\Vt$, $\Pc{i}{j} u\in\Ech{i}{j}$ such that
\begin{equation}
\label{eq:projc}
a_h(u-\Pc{i}{j}u,v)=0\qquad\forall v\in\Ech{i}{j}.
\end{equation}
We shall make use of the Rayleigh quotient associated to the eigenvalue
problem~\eqref{eq:var}
\begin{equation}
\rayl(w)=\frac{a(w,w)}{(w,w)}\quad\forall w\in\Huo\setminus\{0\}
\label{eq:rayl}
\end{equation}
and of the analogous quotient associated to the nonconforming discretization
\begin{equation}
\raylh(w)=\frac{a_h(w,w)}{(w,w)}\quad\forall w\in\Vh\setminus\{0\}
\label{eq:raylh}
\end{equation}
In case of multiple eigenvalues we shall need to estimate the distance between
eigenspaces associated to them and to their discrete counterpart.
Let $E$ and $F$ be two subspaces of $\Vt$, then the distance between them
is defined as 
\[
\delta_h(E,F)=\sup_{\genfrac{}{}{0pt}{}{u\in E}{\|u\|_h=1}} 
\inf_{v\in F}\|u-v\|_h.
\]
For nonzero functions $u$ and $v$, if $E=\spann\{u\}$, we
write $\NEW{\delta_h}(u,F)$ instead of $\NEW{\delta_h}(E,F)$ and if $E=\spann\{u\}$ and
$F=\spann\{v\}$, we write $\NEW{\delta_h}(u,v)$ for $\NEW{\delta_h}(E,F)$.
We have $0\le\delta_h(E,F)\le1$ and $\delta_h(E,F)=0$ if and only if
$E\subseteq F$. If $\dim E=\dim F<\infty$ then $\delta_h(E,F)=\delta_h(F,E)$.
If $P$ and $Q$ are the orthogonal projections onto $E$ and $F$, respectively,
then $\delta_h(E,F)$ equals the largest singular value of the operator $(I-Q)P$
and
\begin{equation}
\label{eq:LGproj}
\delta_h(E,F)=\|(I-Q)P\|_{\L(\Vt)},
\end{equation}
\NEW{where the notation $\|\cdot\|_{\L(\Vt)}$, as usual, denotes the operator
norm from $\Vt$ into itself.}
See for example~\cite{knyazevosborn} for these results and the characterization
of the distance between subspaces.
\section{Error estimates for the eigenfunctions}
In this section we introduce the error indicators and present the a posteriori
error estimates for the eigenfunctions.

First of all let us recall some properties of the Couzeix--Raviart space.
Given $w\in\Huo$, we denote by $w_I\in\Vh$ its edge/face average interpolant
such that
\begin{equation}
\label{eq:Vncinterp}
\int_F w_I=\int_F w\qquad \forall F\in\mathcal{F}_h. 
\end{equation}
It is well known that $\nh w_I$ is the $L^2$-projection of $\nabla w$ onto the
piecewise constant vector fields and that the following estimates hold true
\begin{equation}
\label{eq:PropI}
\aligned
&\|\nh w_I\|_0\le\|\nabla w\|_0,\\
&\|w-w_I\|_{L^2(K)}\le C_1 h_K\|\nabla w\|_{L^2(K)}.
\endaligned
\end{equation}
Our error indicators make use of the
following conforming postprocessing for the elements in $\Vh$.
To any element $v\in\Vh$ we associate an element $\tv\in\Vc$ obtained by
averaging the value of $v$ at the vertices of the triangulation
$\mathcal{T}_h$. Namely, following~\cite{DDP}, 
for each internal vertex $\NEW{\mathsf{P}}$ we consider all elements
$K_i\in\mathcal{T}_h$ for $i=1,\dots,M$ which share the vertex $P$ and define
\begin{equation}
\label{eq:postproc}
\tv(\NEW{\mathsf{P}})=\sum_{i=1}^Mw_i v|_{K_i}(\NEW{\mathsf{P}}),
\end{equation}
where $w_i$ are suitable weights such that $\sum_{i=1}^Mw_i=1$. 
\begin{lemma}
\label{le:stimetildev}
The following estimates hold true with constants $C$ independent of $h$
\begin{equation}
\label{eq:stimetildev}
\aligned
&\|\tv\|_0\le C\|v\|_0\\
&\|\nabla\tv\|_0\le C\|\nabla_h v\|_0\\
&\|\tv-v\|_0\le Ch\|\nh(\tv-v)\|_0.
\endaligned
\end{equation}
\end{lemma}
\begin{proof}
Let $\NEW{\mathsf{P}}$ be a vertex of the mesh. It is proved in~\cite[Th.~5.2]{DDP}, that
for all $w\in\Huo$
\begin{equation}
\label{eq:tildev1}
|\tv(\NEW{\mathsf{P}})-v|_{K}(\NEW{\mathsf{P}})| \le \frac{C}{h_K^{d/2-1}}
\|\nh(v-w)\|_{L^2(\Omega_\NEW{\mathsf{P}})},
\end{equation}
where $\Omega_\NEW{\mathsf{P}}$ is the union of elements $K$ containing
$\NEW{\mathsf{P}}$.

For $w=0$, using an inverse estimate, we also have
\[
|\tv(\NEW{\mathsf{P}})-v|_K(\NEW{\mathsf{P}})| \le \frac{C}{h_K^{d/2}}
\|v\|_{L^2(\Omega_\NEW{\mathsf{P}})},
\]
Then, for an element $K$ we write, using the standard notation $N_i$ for nodal
basis functions,
\[
\tv-v=\sum_{i=1}^{d+1} \left(\tv(\NEW{\mathsf{P}}_i)-v|_K(\NEW{\mathsf{P}}_i)\right)  N_i
\]
and
\[
\nabla(\tv-v)=\sum_{i=1}^{d+1}\left(\tv(\NEW{\mathsf{P}}_i)-v|_K(\NEW{\mathsf{P}}_i)\right) \nabla N_i
\]
and therefore, if $\tilde K$ is the union of neighbors of $K$,
using the above estimates and standard estimates for the basis functions $N_i$,
we obtain
\[
\|\tv-v\|_{0,K}\le C \|v\|_{0,\tilde K}
\]
and
\[
\|\nabla(\tv-v)\|_{0,K}\le C \|\nabla_h v\|_{0,\tilde K}.
\]
Then the triangle inequality yields the first two estimates
in~\eqref{eq:stimetildev}. 
The last one can also be easily obtained from~\eqref{eq:tildev1} taking into
account that $\|N_i\|_0\le C h_K^{d/2}$ and choosing $w=\tv$.

\end{proof}

We define the local and global error estimators as follows
\begin{equation}
\label{eq:stimatore}
\aligned
&\mu^2_{i,\NEW{K}}=\|\nabla \tilde u_{i,h}-\nabla_h u_{i,h}\|^2_{L^2(\NEW{K})},\quad
&&\mu^2_i=\sum_\NEW{K} \mu^2_{i,\NEW{K}}\\
&\eta^2_{i,\NEW{K}}=h_\NEW{K}^2\|\lambda_{i,h}u_{i,h}\|_{L^2(\NEW{K})}^2,
&&\eta^2_i=\sum_\NEW{K} \eta^2_{i,\NEW{K}}.
\endaligned
\end{equation}
The following theorem gives the error estimates for the eigenfunctions
in term of the above error indicators.
\begin{thm}
\label{th:errfunct}
Let $\lambda_i$ be an eigenvalue of~\eqref{eq:var} with multiplicity $q\ge1$
(that is $\lambda_i=\dots=\lambda_{i+q-1}$)
and let $\E{i}{i+q-1}$ be the associated eigenspace. Assume that
$\lambda_{j,h}$ is a discrete eigenvalue of~\eqref{eq:pbnc}
converging to $\lambda_i$ and that $E_{j,h}$ is the associated eigenspace
($j=i,\dots,i+q-1$).
Then 
\begin{equation}
\NEW{\delta_h}(E_{j,h},\E{i}{i+q-1})\le\NEW{\delta_h}(E_{j,h},\Huo)+
C_1\frac{\eta_j}{\lambda_{j,h}}+h.o.t.
\label{eq:errfunct}
\end{equation}
More precisely, we have
\[
|h.o.t.|\le\frac{C_{\Omega}}{\lambda_{j,h}}
\left((\lambda_i-\lambda_{j,h})+(\lambda_i\lambda_{j,h})^{1/2}
\inf_{v\in\E{i}{i+q-1}}\|v-u_{j,h}\|_{\Ld}\right),
\]
where $C_1$ is the constant in~\eqref{eq:PropI} and $C_\Omega$ is the
Poincar\'e constant. 
\end{thm}
\begin{proof}
The proof is based on that of~\cite[Th.~3.2]{DDP}. Here we make more
precise the case of multiple eigenvalues.

Let us fix $j=i,\dots,i+q-1$ and let us consider the eigensolution
$(\lambda_{j,h},u_{j,h})$ of~\eqref{eq:pbnc}; we recall that
$\|u_{j,h}\|_0=1$.
Let $u_j(h)\in\E{i}{i+q-1}$ be such that $\|u_j(h)\|_0=1$ and
\begin{equation}
\label{eq:ujh}
\|u_j(h)-u_{j,h}\|_h=\inf_{v\in\E{i}{i+q-1}}\|v-u_{j,h}\|_h.
\end{equation}
We observe that $u_j(h)$ is an eigenfunction associated to the multiple
eigenvalue $\lambda_i$, hence it satisfies~\eqref{eq:var}.
Then applying the same argument as in the proof
of~\cite[Th.~3.2]{DDP}, we have that
\begin{equation}
\label{eq:ujh2}
\aligned
\|u_j(h)-u_{j,h}\|_h\le&\inf_{v\in\Huo}\|v-u_{j,h}\|_h+C_1\eta+
C_\Omega\Big((\lambda_i-\lambda_{j,h})\\
&+(\lambda_i\lambda_{j,h})^{1/2}\|u_j(h)-u_{j,h}\|_0\Big).
\endaligned
\end{equation}
From the definition of the gap we have to estimate
\[
\NEW{\delta_h}(E_{j,h},\E{i}{i+q-1})=
\sup_{\genfrac{}{}{0pt}{}{u\in E_{j,h}}{\|u\|_h=1}} 
\inf_{v\in \E{i}{i+q-1}}\|u-v\|_h=
\inf_{v\in \E{i}{i+q-1}}\Big\|\frac{u_{j,h}}{\|u_{j,h}\|_h}-v\Big\|_h,
\]
since $E_{j,h}$ is generated by $u_{j,h}$. With a simple computation, using
the above estimate for the eigenfunction $u_{j,h}$ and the fact that
$\|u_{j,h}\|_h=\lambda_{j,h}$, we obtain the desired bound. 
\end{proof}
It remains to estimate the gap between $E_{j,h}$ and $\Huo$ in terms of our
indicators.
\begin{lemma}
\label{le:gapHuo}
Under the same assumptions as in Theorem~\ref{th:errfunct}, the following estimate
holds true:
\begin{equation}
\label{eq:gapHuo}
\NEW{\delta_h}(E_{j,h},\Huo)\le\frac{1}{\lambda_{j,h}}\mu_j.
\end{equation}
\end{lemma}
\begin{proof}
By definition we have
\[
\aligned
\NEW{\delta_h}(E_{j,h},\Huo)&=
\sup_{\genfrac{}{}{0pt}{}{u\in E_{j,h}}{\|u\|_h=1}} 
\inf_{v\in\Huo}\|u-v\|_h\\
&=\frac{1}{\lambda_{j,h}}\inf_{v\in\Huo}\|u_{j,h}-v\|_h\\
&\le\frac{1}{\lambda_{j,h}}\|u_{j,h}-\tilde u_{j,h}\|_h=
\frac{1}{\lambda_{j,h}}\mu_j.
\endaligned
\]
\end{proof}
For the efficiency of these error estimators we refer to~\cite{DDP} where the
following local bounds from below of the error are proved.
\begin{thm}
\label{th:efficiency}
Let $\lambda_i$ be an eigenvalue of~\eqref{eq:var} with multiplicity $q\ge1$
and let $\lambda_{j,h}$ be a discrete eigenvalue converging to $\lambda_i$
($j=i,\dots,i+q-1$).
Let $u_j(h)\in\E{i}{i+q-1}$ be such that~\eqref{eq:ujh} holds true. Then
there exist constants $C$ depending only on the regularity of the elements
such that for all elements $K\in\mathcal{T}_h$ it holds
\[
\aligned
&\mu_K\le C\|\nh(u_j(h)-u_{j,h})\|_{L^2(K^*)}\\
&\eta_K\le C\|\nh(u_j(h)-u_{j,h})\|_{L^2(K)}+h.o.t.,
\endaligned
\]
where $K^*$ is the union of all the elements in $\mathcal{T}_h$ sharing a
vertex with $K$ and
\[
h.o.t.=h_K\|\lambda_iu_j(h)-\lambda_{j,h}u_{j,h}\|_{L^2(K)}.
\]
\end{thm}
\section{Error estimates for the eigenvalues}
In this section we prove error estimates for the eigenvalues using the
a posteriori error indicators introduced in~\eqref{eq:stimatore}.
In the case of conforming approximation of the eigenvalue
problem~\eqref{eq:var} it is well known that each discrete eigenvalue is
greater than or equal to the corresponding continuous one. In the case of
nonconforming discretization this is not true in general. In~\cite{AD,DDP} it
is proved that, for singular eigenfunctions, the Crouzeix-Raviart approximation
provides asymptotic lower bounds of the corresponding eigenvalue.
For this reason, in our analysis we consider separately the cases where a
multiple eigenvalue is approximated by below or by above. More precisely,
given a multiple eigenvalue $\lambda_i$ of multiplicity $q\ge1$, we assume
that either $\lambda_{j,h}\le\lambda_i$ or $\lambda_i\le\lambda_{j,h}$ for all
$j=i,\dots,i+q-1$.

Let us consider first the case when the eigenvalues are approximated from
below.

The first theorem gives an estimate of the relative error for the eigenvalues
in terms of the norm of the distance of the discrete eigenspace from the subspace
of conforming finite elements orthogonal to the span of the first $i$
conforming eigenfunctions.

\begin{thm}
\label{th:main}
Let $\lambda_i$ be an eigenvalue with multiplicity $q$ so that
\[
\lambda_{i-1}<\lambda_i=\dots=\lambda_{i+q-1}<\lambda_{i+q},
\]
and let $\lambda_{i,h}\le\dots\le\lambda_{i+q-1,h}$ be the $q$ discrete
eigenvalues converging to $\lambda_i$.
We assume that $\lambda_{j,h}\le\lambda_i$ for $j=i,\dots,i+q-1$.
Then
\begin{equation}
\label{eq:stima1}
\frac{\lambda_i-\lambda_{j,h}}{\lambda_i}\le
\|(I-\Pch+\Pc{1}{i-1})\Ph{i}{j}\|^2_{\L(\Vt)}.
\end{equation}
\end{thm}

\begin{proof}
We observe that the discretization by conforming finite elements
produces $q$ discrete eigenvalues converging to
$\lambda_i$ and that it holds $\lambda_i\le\lc_{j,h}$ for $j=i,\dots,i+q-1$.

Let us fix $j$ with $i\le j\le i+q-1$, then by assumption we have
\[
\lambda_{j,h}\le\lambda_i\le\lc_{j,h}.
\]

The operators $I-\Pch+\Pc{1}{i-1}$ and $\Ph{i}{j}$ are orthogonal projections
with respect to the norm $\|\cdot\|_h$ of $\Vt$. Therefore
$\|(I-\Pch+\Pc{1}{i-1})\Ph{i}{j}\|_{\L(\Vt)}\le1$
(see~\cite[Th.~6.34, p.~56]{Kato}). If
$\|(I-\Pch+\Pc{1}{i-1})\Ph{i}{j}\|_{\L(\Vt)}=1$ then the bound~\eqref{eq:stima1} is
obviously true, since $\lambda_{j,h}\le\lambda_i$. Hence we assume that
\[
\|(I-\Pch+\Pc{1}{i-1})\Ph{i}{j}\|_{\L(\Vt)}<1.
\]
Thanks to~\cite[Th.~3.6, Chap.~I]{Kato} this inequality implies that
\[
\dim((\Pch-\Pc{1}{i-1})\Eh{i}{j})=\dim(\Eh{i}{j})=j-i+1.
\]

We choose  $\overline u\in (\Pch-\Pc{1}{i-1})\Eh{i}{j}\subset\Vc$ such that
$\|\overline u\|_h=\|\overline u\|_{\NEW{1}}=1$ and
\[
\rayl(\overline u)=
\max_{\genfrac{}{}{0pt}{}{w\in (\Pch-\Pc{1}{i-1})\Eh{i}{j}}{w\ne0}}\rayl(w),
\]
where $\rayl(w)$ is the Rayleigh quotient defined in~\eqref{eq:rayl}.

Let us consider the following orthogonal decomposition of $\overline u$ in
$\Vt$:
\[
\overline u=u+v\quad \text{with }u\in \Eh{1}{j}
\text{ and }v\in(\Eh{1}{j})^\perp,
\]
that is $a_h(w,v)=0$ for all $w\in\Eh{1}{j}$. Notice that since
$u\in\Vh$ also $v\in\Vh$.
We have that
\[
\aligned
\|v\|_h&=\delta_h(\overline u,\Eh{1}{j})&&\text{by definition of }v\\
&\le\delta_h((\Pch-\Pc{1}{i-1})\Eh{i}{j},\Eh{1}{j})&&
\overline u\in (\Pch-\Pc{1}{i-1})\Eh{i}{j}\\
&\le\delta_h((\Pch-\Pc{1}{i-1})\Eh{i}{j},\Eh{i}{j})&&\text{the
inf is taken on a smaller subset}\\
&=\|(I-\Pch+\Pc{1}{i-1})\Ph{i}{j}\|_{\L(\Vt)}&&\text{this is a characterization
of the gap.}
\endaligned
\]
We now prove that
\[
\aligned
&0\le\frac{\lc_{j,h}-\lambda_{j,h}}{\lc_{j,h}}\le\|v\|_h\\
&0\le\frac{\lambda_i-\lambda_{j,h}}{\lambda_{i}}\le\|v\|_h.
\endaligned
\]
We observe that the first inequality implies the second one.\\
By definition of $\overline u$ and the min-max principle for the eigenvalues
we have that
\[
\lc_{j,h}\le\rayl(\overline u).
\]
Moreover, since $u\in\Eh{1}{j}$, we have that $u=\sum_{s=1}^j\alpha_s u_{s,h}$
and
\[
\raylh(u)=\frac{a_h(u,u)}{(u,u)}
=\frac{\sum_{s=1}^j\alpha_s^2a_h(u_{s,h},u_{s,h})}
{\sum_{s=1}^j\alpha_s^2(u_{s,h},u_{s,h})}=
\frac{\sum_{s=1}^j\alpha_s^2\lambda_{s,h}}{\sum_{s=1}^j\alpha_s^2}
\le \lambda_{j,h}.
\]
In conclusion, the following inequalities hold true ($i\le j\le i+q-1$):
\[
\raylh(u)\le\lambda_{j,h}\le\lambda_j\le\lc_{j,h}\le\rayl(\overline u),
\]
and the rest of the proof is based on a bound for
$1/\raylh(u)-1/\rayl(\overline u)$.

Since $v\in(\Eh{1}{j})^\perp$ we have that $a_h(u,v)=0$.
We want to show that also $(u,v)=0$.
Since we know that $v\in\Vh$ and that
$\Eh{1}{j}$ is invariant with respect to $T_h$, we have also $a_h(T_hu,v)=0$.
Hence
\[
0=a_h(T_hu,v)=(u,v),
\]
due to the definition of $T_h$.

We now compute
\[
\aligned
\frac1{\raylh(u)}&-\frac1{\rayl(\overline u)}
=\frac{(u,u)}{a_h(u,u)}-\frac{(u,u)+(v,v)}{a_h(u,u)+a_h(v,v)}\\
&=\frac{a_h(u,u)(u,u)+a_h(v,v)(u,u)-a_h(u,u)(u,u)-\NEW{a_h}(u,u)(v,v)}
{a_h(u,u)\left(a_h(u,u)+a_h(v,v)\right)}\\
&=\frac1{a_h(u,u)+a_h(v,v)}
\left(\frac{a_h(v,v)}{a_h(u,u)}\left((u,u)+(v,v)\right)
-\left(\frac{a_h(v,v)}{a_h(u,u)}+1\right)(v,v)\right)\\
&=\left(\frac1{\rayl(\overline u)}-\frac1{\raylh(v)}\right)
\frac{a_h(v,v)}{a_h(u,u)}\\
&\le \frac1{\rayl(\overline u)}\frac{a_h(v,v)}{a_h(u,u)}\le
\frac1{\lambda_j}\frac{a_h(v,v)}{a_h(u,u)}.
\endaligned
\]
We get
\[
\frac1{\lambda_{jh}}-\frac1{\lambda_j}\le
\frac1{\raylh(u)}-\frac1{\rayl(\overline u)}\le
\frac1{\lambda_j}\frac{a_h(v,v)}{a_h(u,u)}
\]
from which we obtain
\[
\frac{\lambda_j}{\lambda_{j,h}}\le1+\frac{a_h(v,v)}{a_h(u,u)}=
\frac{\NEW{a_h}(u,u)+a_h(v,v)}{a_h(u,u)}=\frac1{a_h(u,u)}
\]
and then
\[
\frac{\lambda_j-\lambda_{j,h}}{\lambda_j}=1-\frac{\lambda_{j,h}}{\lambda_j}
\le 1-a_h(u,u)=a_h(v,v)=\|v\|_h^2.
\]
We can obtain also
\[
\frac{\lc_{j,h}-\lambda_{j,h}}{\lc_{j,h}}\le a_h(v,v)=\|v\|_h^2,
\]
by using the following inequality
\[
\frac1{\lambda_{j,h}}-\frac1{\lc_{j,h}}\le
\frac1{\raylh(u)}-\frac1{\rayl(\overline u)}\le
\frac1{\lc_{j,h}}\frac{a_h(v,v)}{a_h(u,u)}.
\]
\end{proof}
We now want to estimate the right hand side of~\eqref{eq:stima1} in
terms of $\Pch$ and $\Ph{i}{j}$ only.

First of all, we observe that
\[
\|(I-\Pch+\Pc{1}{i-1})\Ph{i}{j}\|^2_{\L(\Vt)}=
\|(I-\Pch)\Ph{i}{j}\|^2_{\L(\Vt)}+\|\Pc{1}{i-1}\Ph{i}{j}\|^2_{\L(\Vt)};
\]
hence it remains to estimate the second term, which represents the projection
of the nonconforming invariant subspace associated to the
eigenvalues numbered from $i$ to $j$ onto the subspace of conforming invariant
subspace generated by the first $i-1$ eigenvalues.
\begin{proposition}
\label{pr:knyazev}
Let $\lambda_i$ be an eigenvalue with multiplicity $q$, so that
\[
\lambda_{i-1}<\lambda_i=\dots=\lambda_{i+q-1}<\lambda_{i+q},
\]
and let $\lambda_{i,h}\le\dots\le\lambda_{i+q-1,h}$ be the $q$ discrete
eigenvalues converging to $\lambda_i$.
We assume that
\NEW{$\lc_{i-1,h}<\lambda_{i,h}$ then,}
\NEW{for $h$ small enough, there exists $\beta>0$ such that}
\begin{equation}
\label{eq:stima2}
\|\Pc{1}{i-1}\Ph{i}{j}\|_{\L(\Vt)}\le
\frac{\|(I-\Pch)T_h\Pc{1}{i-1}\|_{\L(\Vt)}}{\NEW{\beta}}
\|(I-\Pch)\Ph{i}{j}\|_{\L(\Vt)}
\end{equation}
\NEW{where $T_h$ is the solution operator defined in~\eqref{eq:soluz}}.
\end{proposition}
\begin{proof}
Using the same notation as in~\cite[Th.~4.2]{knyazev}, we introduce the
following operators:
\[
P=\Ph{i}{j},\quad \tilde R=\Pc{1}{i-1},
\]
so that $P$ is the elliptic projection onto the
invariant nonconforming subspace $E_{i,\dots,j,h}$, and
$\tilde R$ is the elliptic projection onto the invariant conforming
subspace $E^c_{1,\dots,i-1,h}$, that is
\[
\aligned
&\tilde R=\Pc{1}{i-1}:\Vt\to E^c_{1,\dots,i-1,h}\\
& a(\tilde Rw,v)=a_h(w,v)\quad\forall v\in E^c_{1,\dots,i-1,h}.
\endaligned
\]
We observe that $\|\tilde R w\|_1\le \|w\|_h$

Moreover, the spectrum of $(\tilde RT_h\tilde R)|_{\im\tilde R}$ is equal to
$\{1/\lc_{1,h},\dots,1/\lc_{i-1,h}\}$. Indeed, $\im\tilde R$ is the span of
$\{\uc_{1,h},\dots,\uc_{i-1,h}\}$ and by definition $T_h\uc_{k,h}$ ($1\le k\le
i-1$) belongs to $\Vh$ and is given by
\[
a_h(T_h\uc_{k,h},v)=(\uc_{k,h},v)\quad\forall v\in\Vh.
\]
Hence, $\tilde RT_h\uc_{k,h}$ belongs to $E^c_{1,\dots,i-1,h}$ and satisfies
\[
a(\tilde RT_h\uc_{k,h},v)=a_h(T_h\uc_{k,h},v)=(\uc_{k,h},v)\quad\forall v\in
E^c_{1,\dots,i-1,h}.
\]
It follows that
\[
\tilde RT_h\uc_{k,h}=\frac1{\lc_{k,h}}\uc_{k,h},
\]
so that $1/\lc_{k,h}$ ($1\le k\le i-1$) coincides with the spectrum of
$(\tilde RT_h\tilde R)|_{\im\tilde R}$ (there cannot be other eigenvalues,
since the dimension of $\im\tilde R$ is equal to $i-1$).

Since the spectrum of $(\tilde RT_h\tilde R)|_{\im\tilde R}$ does not
contain the eigenvalues $\NEW{\nu}_{j,h}=1/\lambda_{j,h}$ for $j=i,\dots,i+q-1$,
the operator $\tilde R(T_h-\NEW{\nu}_{j,h})\tilde R$ has a bounded inverse and
\[
d\|\tilde RP\|_{\L(\Vt)}\le\|\tilde R(T_h-\NEW{\nu}_{j,h})\tilde RP\|_{\L(\Vt)},
\]
where
\[
d=\min_{k=1,\dots,i-1}|\NEW{\nu}^c_{k,h}-\NEW{\nu}_{j,h}|=|\NEW{\nu}^c_{i-1,h}-\NEW{\nu}_{j,h}|
\NEW{\,\ge\,} |\NEW{\nu}^c_{i-1,h}-\NEW{\nu}_{i,h}|=
\frac{\lambda_{i,h}-\lc_{i-1,h}}{\lambda_{i,h}\lc_{i-1,h}}.
\]
We have that
\begin{equation}
\label{eq:remark}
\tilde R(T_h-\NEW{\nu}_{j,h})\tilde RP=\tilde R(T_h-\NEW{\nu}_{j,h})\Pch P.
\end{equation}
Namely, since $\tilde R\tilde R P=\tilde R\Pch P$, it is enough to show that
$\tilde RT_h\tilde RP=\tilde RT_h\Pch P$. We have that $\Thc=\Pch T_h$, which
implies that $\tilde R T_h\tilde RP=\tilde R\Thc\tilde RP$ and that
$\tilde RT_h\Pch P=\tilde R\Thc\Pch P$. Hence, we only have to show that
$\tilde R\Thc\tilde RP=\tilde R\Thc\Pch P$. Indeed, it holds
$\tilde R\Thc\tilde R=\tilde R\Thc\Pch$. In order to show this result,
let's take $f\in\Vh$, then $\Pch f=\sum_{i=1}^{\dim\Vc}\alpha_i u^c_{i,h}$ and
$\tilde R f=\sum_{i\in I}\alpha_i u^c_{i,h}$, where $I$ is the finite set of
indices corresponding to the range of $\tilde R$. The equality~\eqref{eq:remark}
is then easily obtained by comparing $\tilde R\Thc\tilde Rf$ and
$\tilde R\Thc\Pch f$ and taking into account that $u_{i,h}^c$ are
eigenfunctions of $\Thc$.

From~\eqref{eq:remark} we obtain
\[
\aligned
\tilde R(T_h-\NEW{\nu}_{j,h})\Pch P&=
\tilde RT_h\Pch P-\tilde RT_h P+\tilde RT_h P-\tilde R\NEW{\nu}_{j,h}\Pch P\\
&=-\tilde RT_h(I-\Pch)P+\tilde R(T_h-\NEW{\nu}_{j,h}) P+\NEW{\nu}_{j,h}\tilde R(I-\Pch) P.
\endaligned
\]
The last term is equal to zero since $\tilde R=\tilde R\Pch$.
Hence
\[
\aligned
d\|\tilde R P\|_{\L(\Vt)}&\le\|\tilde RT_h(I-\Pch)P\|_{\L(\Vt)}+
\|\tilde R(T_h-\NEW{\nu}_{j,h}) P\|_{\L(\Vt)}\\
&\le\|\tilde RT_h(I-\Pch)\|_{\L(\Vt)}\|(I-\Pch) P\|_{\L(\Vt)}
+
\|\tilde R P\|_{\L(\Vt)}
\|P(T_h-\NEW{\nu}_{j,h})P\|_{\L(\Vt)}
\\
&\le \|(I-\Pch)T_h\tilde R\|_{\L(\Vt)}\|(I-\Pch)P\|_{\L(\Vt)}
+\delta\|\tilde R P\|_{\L(\Vt)},
\endaligned
\]
\NEW{where $\delta$ is given by}
\[
\NEW{\delta=
\|P(T_h-\NEW{\nu}_{j,h})P\|_{\L(\Vt)}.}
\]
Since $P$ is the elliptic projection onto the
invariant nonconforming subspace $E_{i,\dots,j,h}$ we have that
\[
\delta\le|\NEW{\nu}_{i+q-1,h}-\NEW{\nu}_{i,h}|=
\frac1{\lambda_{i,h}}-\frac1{\lambda_{i+q-1,h}}=
\frac{\lambda_{i+q-1,h}-\lambda_{i,h}}{\lambda_{i,h}\lambda_{i+q-1,h}}
\le\frac{\lambda_{i+q-1,h}-\lambda_{i,h}}{\lambda_{i,h}\lc_{i-1,h}}.
\]
For $h$ small enough, $\NEW{\beta=d-\delta}>0$ and we conclude that
\[
\|\tilde R P\|_{\L(\Vt)}\le
\frac{\|(I-\Pch)T_h\tilde R\|_{\L(\Vt)}}{\NEW{\beta}}\|(I-\Pch)P\|_{\L(\Vt)}.
\]
\end{proof}
Combining the results of Theorem~\ref{th:main} and of
Proposition~\ref{pr:knyazev} we have the following result
\begin{equation}
\label{eq:stimasotto}
\frac{\lambda_i-\lambda_{j,h}}{\lambda_i}\le
\left(1+\frac{\|(I-\Pch)T_h\Pc{1}{i-1}\|_{\L(\Vt)}}{\NEW{\beta}}\right)
\|(I-\Pch)\Ph{i}{j}\|^2_{\L(\Vt)},
\end{equation}
from which we deduce the following a posteriori estimate involving the
indicators introduced in~\eqref{eq:stimatore}
\begin{thm}
\label{th:stimasotto}
Let us assume the same hypotheses as in Theorem~\ref{th:main}
and Proposition~\ref{pr:knyazev}. Then, for 
$h$ small enough, we have
\[
\frac{\lambda_i-\lambda_{j,h}}{\lambda_i}\le C \NEW{\delta_h^2}(\Eh{i}{j},\Vc)
\le C \sum_{k=i}^j\frac1{\lambda^2_{k,h}}\mu_k^2.
\]
\end{thm}
\begin{proof}
The quotient within the parentheses in~\eqref{eq:stimasotto} tends to zero as $h$
tends to zero, hence it is bounded. On the other hand, from~\eqref{eq:LGproj}
we have that
\[
\|(I-\Pch)\Ph{i}{j}\|^2_{\L(\Vt)}=\NEW{\delta_h^2}(\Eh{i}{j},\Vc).
\]
Thanks to~\eqref{eq:ortoh}, $u_{k,h}$ for $k=i,\dots,j$ form an orthogonal
basis for $\Eh{i}{j}$, so that (see, e.g.~\cite[Cor.~2.2]{knyazev})
\[
\NEW{\delta_h^2}(\Eh{i}{j},\Vc)\le\sum_{k=i}^j\NEW{\delta_h^2}(E_{k,h},\Vc).
\]
Applying Lemma~\ref{le:gapHuo} we arrive at the desired estimate.
\end{proof}
Let us now consider the case of discrete nonconforming eigenvalues
approximating the continuous ones from above. We estimate first the distance
between an eigenvalue $\lambda_i$ of multiplicity $q$ and the average of
the discrete eigenvalues $\lambda_{k,h}$ for $k=i,\dots,j$ (here
$j=i,\dots,i+q-1$).
\begin{lemma}
\label{le:main2}
Let $\lambda_i$ be an eigenvalue with multiplicity $q$, so that
\[
\lambda_{i-1}<\lambda_i=\dots=\lambda_{i+q-1}<\lambda_{i+q},
\]
and let $\lambda_{i,h}\le\dots\le\lambda_{i+q-1,h}$ be $q$ discrete
eigenvalues converging to $\lambda_i$.
We assume that $\lambda_{j,h}\ge\lambda_i$ for $j=i,\dots,i+q-1$, then
for all $j=i,\dots,i+q-1$
\begin{equation}
\label{eq:stimamedia}
\frac1{J}\sum_{k=i}^{j}\lambda_{k,h}-\lambda_i\le
\frac1{J}\sum_{k=i}^{j}\left(6\mu_k^2+4C_1\eta_k^2+4|h.o.t.|^2\right)
+|h.o.t.|_2
\end{equation}
where
\[
|h.o.t.|_2=\frac{C}{J}\sum_{k=i}^{j}\sum_{m=1}^{k-i}h^{2m}\sum_{l=i}^{k-m}
C(\mu_k+\mu_l)^2\lambda_{l,h},
\]
$J=j-i+1$ and the higher order terms $|h.o.t.|$ are defined in
Th.~\ref{th:errfunct}.
\end{lemma}
\begin{proof}
The proof is divided into two parts. We start by estimating the error for the
first discrete eigenvalue $\lambda_{i,h}$ converging to $\lambda_i$, next we
shall deal with the general case.

{\bf First case.} Let $\lambda_i$ with $i\ge1$ be a multiple eigenvalue with
multiplicity $q\ge1$.

It holds that $\lambda_i\le\lambda_{i,h}\le\lc_{i,h}$.
The first inequality holds by assumption and the second one is due to the
min-max principle for the eigenvalues and the fact that $\Vc\subset\Vh$.
Let $u_{i,h}\in V_h$ be the eigensolution associated with $\lambda_{i,h}$
and $u_i(h)\in E_{i,\dots,i+q-1}$
be such that $\|u_i(h)\|_0=1$ and satisfies~\eqref{eq:ujh} and~\eqref{eq:ujh2},
then, for all $v\in(\Pch-\Pc{1}{i-1})E_{i,h}$ with $\|v\|_0=1$, we have
\[
\lambda_i+\lambda_{i,h}\le\|\nabla u_i(h)\|^2_0+\rayl(v)=
\|\nabla u_i(h)\|^2_0+\|\nabla v\|^2_0.
\]
Hence
\[
\aligned
\lambda_i+\lambda_{i,h}&\le\|\nabla u_i(h)\|^2_0+\|\nabla v\|^2_0\\
&=\|\nabla(u_i(h)-v)\|^2_0+2\int_\Omega\nabla u_i(h)\nabla v\,dx\\
&=\|\nabla(u_i(h)-v)\|^2_0+2\lambda_i\int_\Omega u_i(h) v\,dx\\
&=\|\nabla(u_i(h)-v)\|^2_0-\lambda_i\|u_i(h)-v\|^2_0+2\lambda_i
\endaligned
\]
and
\[
\lambda_{i,h}-\lambda_i\le\|\nabla(u_i(h)-v)\|^2_0\le
2\|\nabla_h(u_i(h)-u_{i,h})\|^2_0+2\|\nabla_h(u_{i,h}-v)\|^2_0.
\]
The first term on the right hand side can be estimated with~\eqref{eq:ujh2}.
Since $v$ is arbitrary in $(\Pch-\Pc{1}{i-1})E_{i,h}$,
we can set
\NEW{$v=(\Pch-\Pc{1}{i-1})u_{i,h}$. Hence
\[
\aligned
\|\nabla_h(u_{i,h}-v)\|_0&=\|\nabla_h(u_{i,h}-(\Pch-\Pc{1}{i-1})u_{i,h})\|_0\\
&\le\|\nabla_h(I-\Pch)u_{i,h}\|_0+\|\nabla_h\Pc{1}{i-1}u_{i,h}\|_0.
\endaligned
\]
The hypotheses of Proposition~\ref{pr:knyazev} are satisfied, so that
\[
\|\nabla_h(u_{i,h}-v)\|_0\le C\|\nabla_h(I-\Pch)u_{i,h}\|_0\le C\mu_i.
\]
Finally,} we obtain 
\[
\lambda_{i,h}-\lambda_i\le4\mu_i^2+C_1\eta_i^2+|h.o.t.|^2.
\]
{\bf Second case.} 
Due to the min-max characterization for the eigenvalues we have
that for $k=i,\dots,j$
\[
\lambda_i\le\lambda_{k,h}\le\lc_{k,h}.
\]

Let us take $J$ arbitrary elements $v_1,\dots,v_J$ of $\Vc$ belonging to
$(E_{1,\dots,i-1,h}^c)^\perp$ \NEW{(to be specified later in the proof)}
which are orthogonal to each
other, that is $(v_l,v_m)=0$ for $l\ne m$.

We show that
\begin{equation}
\sum_{k=i}^{j}\lc_{k,h}\le\sum_{l=1}^{J}\rayl(v_l),
\label{eq:bella}
\end{equation}
where $\rayl(v_l)$ denotes the Rayleigh quotient (see~\eqref{eq:rayl}).
For simplicity, we show the result under the normalization
$\|v_l\|_0=1$ ($l=1,\dots,J$), so that $\rayl(v_l)=a(v_l,v_l)$ (this does not
limit the generality of our proof).
Let us write $v_l$ \NEW{by decomposing it as its components in
$E^c_{i,\dots,j,h}$ and a remainder as follows:}
\[
v_l=\sum_{k=i}^{j}\alpha_{lk}\uc_{k,h}+w_l.
\]
It turns out that
\[
\rayl(v_l)=a(v_l,v_l)\ge\sum_{k=i}^{j}\alpha_{lk}^2\lc_{k,h}+
\left(1-\sum_{k=i}^{j}\alpha_{lk}^2\right)\lc_{j+1,h},
\]
so that
\[
\sum_{l=1}^J\rayl(v_l)\ge\sum_{k=i}^{j}\lc_{k,h}+
\sum_{k=i}^{j}(\lc_{j+1,h}-\lc_{k,h})
\left(1-\sum_{l=1}^J\alpha_{lk}^2\right).
\]
Hence, the estimate~\eqref{eq:bella} is proved if we can show that
\[
\sum_{l=1}^J\alpha_{lk}^2\le1\quad\forall k.
\]
This follows from the orthogonality of the $v_l$'s. Indeed, for all $k$,
\[
\sum_{l=1}^J\alpha_{lk}^2
=\Big(\sum_{l=1}^J\alpha_{lk}v_l,\uc_{k,h}\Big)\le
\Big\Vert\sum_{l=1}^J\alpha_{lk}v_l\Big\Vert_0=
\sqrt{\sum_{l=1}^J\alpha_{lk}^2}.
\]
As in the proof of the first case, for $k=i,\dots,j$,
let $u_k(h)\in E_{i,\dots,i+q-1}$ with $\|u_k(h)\|_0=1$
be such that~\eqref{eq:ujh2} holds true. Then we have
\[
\aligned
J\lambda_i+\sum_{k=i}^{j}\lambda_{k,h}&\le
\sum_{k=i}^{j}\left(\|\nabla u_k(h)\|^2_0+\|\nabla v_{k-i+1}\|^2_0\right)\\
&=\sum_{k=i}^{j}
\left(\|\nabla(u_k(h)-v_{k-i+1})\|^2_0-\lambda_i\|u_k(h)-v_{k-i+1}\|^2_0\right)
+2J\lambda_i
\endaligned
\]
Hence
\begin{equation}
\label{eq:stima3}
\aligned
\frac1J\sum_{k=i}^{j}\lambda_{k,h}-\lambda_i&\le
\frac1J\sum_{k=i}^{j}\|\nabla(u_k(h)-v_{k-i+1})\|^2_0\\
&\le\frac2J\sum_{k=i}^{j}
\left(\|\nabla_h(u_k(h)-u_{k,h})\|^2_0+\|\nabla_h(u_{k,h}-v_{k-i+1})\|^2_0\right)
\endaligned
\end{equation}

The first term in the sum appearing in the last line of~\eqref{eq:stima3} can
be bounded by~\eqref{eq:ujh2}, while for the remaining term we \NEW{now make a
choice for the definition of
$v_l$, $l=1,\dots,J$, which will be analogous to what has been done in the
first case.}

For $k=i,\dots,j$ let \NEW{$\hat u_{k,h}\in\Vc$ be the projection of
$u_{k,h}$ onto $(\Ech{1}{i-1})^\perp$ performed according to~\eqref{eq:defPch}.}

Then we choose
\begin{equation}
\label{eq:defv}
\aligned
&v_1=\NEW{\hat u_{i,h}}\\
&v_l=\NEW{\hat u_{i+l-1,h}}-
\sum_{m=1}^{l-1}\frac{(\NEW{\hat u_{i+l-1,h}},v_m)}{\|v_m\|^2_0}v_m,
\qquad l=2,\dots,J.
\endaligned
\end{equation}
By construction, we have that for $h$ small enough
\begin{equation}
\label{eq:stimav}
\frac12\le\|v_l\|_0\le C.
\end{equation}
\NEW{The same argument used in the first case shows that
\begin{equation}
\label{eq:stimav1}
\|\nabla_h(u_{k,h}-\hat u_{k,h})\|^2_0\le C\mu^2_k,\quad\forall k=i,\dots,j.
\end{equation}
Moreover, standard properties of the projections and a duality argument give
\begin{equation}
\|u_{k,h}-\hat u_{k,h}\|_0\le Ch\mu_k,\quad\forall k=i,\dots,j.
\label{eq:dualproj}
\end{equation}
}%
\NEW{In order to estimate the right hand side of~\eqref{eq:stima3}, we start
with $j=i$; we have
\begin{equation}
\label{eq:nv1}
\|\nabla v_1\|^2_0=\|\nabla\hat u_{i,h}\|^2_0\le\|\nh u_{i,h}\|^2_0=
\lambda_{i,h}.
\end{equation}
}%
For $k=i+1,\dots,j$, set $l=k-i+1$, so that we have
\begin{equation}
\label{eq:stimav2}
\|\nabla_h(u_{k,h}-v_l)\|^2_0\le\|\nabla_h(u_{k,h}-\NEW{\hat u_{k,h}})\|^2_0
+\sum_{m=1}^{l-1}\frac{|(\NEW{\hat u_{k,h}},v_m)|^2}{\|v_m\|_0^4}
\|\nabla v_m\|_0^2.
\end{equation}
We detail the cases when $k=i+1$ and $k=i+2$; the general
situation should then be clear.

If $k=i+1$, then $l=2$, so that the last term in~\eqref{eq:stimav2}
can be estimated using~\eqref{eq:stimav} \NEW{and~\eqref{eq:dualproj}} as follows:
\begin{equation}
\aligned
|(\NEW{\hat u_{i+1,h}},v_1)|&\le|(\NEW{\hat u_{i+1,h}}-u_{i+1,h},v_1)|
+|(u_{i+1,h},v_1-u_{i,h})|\\
&\le \|\NEW{\hat u_{i+1,h}}-u_{i+1,h}\|_0\|v_1\|_0
+\|u_{i+1,h}\|_0\|\NEW{\hat u_{i,h}}-u_{i,h}\|_0\\
&\le Ch(\mu_{i+1}+\mu_i).
\endaligned
\label{eq:stimal2}
\end{equation}
Hence we obtain
\[
\aligned
\|\nabla_h(u_{i+1,h}-v_2)\|^2_0&\le\|\nabla_h(u_{i+1,h}-\NEW{\hat u_{i+1,h}})\|^2_0
+Ch^2(\mu_{i+1}^2+\mu_i^2)\lambda_{i,h}\\
&\NEW{\,\le C}\mu_{i+1}^2+Ch^2(\mu_{i+1}^2+\mu_i^2)\lambda_{i,h}
\endaligned
\]
and
\begin{equation}
\label{eq:nv2}
\|\nabla v_2\|^2_0\le2\|\nabla \NEW{\hat u_{i+1,h}}\|^2_0+2
\frac{|(\NEW{\hat u_{i+1,h}},v_1)|^2}{\|v_1\|_0^4}\|\nabla v_1\|^2_0\le
C\lambda_{i+1,h}+Ch^2(\mu_{i+1}+\mu_i)^2\lambda_{i,h}.
\end{equation}
If $k=i+2$, then $l=3$ and we have two terms in the sum on the right hand side
of~\eqref{eq:stimav2}, that is
$(\NEW{\hat u_{i+2,h}},v_1)$ and $(\NEW{\hat u_{i+2,h}},v_2)$.
For the first term, working as in~\eqref{eq:stimal2}, we obtain
\begin{equation}
\aligned
|(\NEW{\hat u_{i+2,h}},v_1)|&\le C
\|\NEW{\hat u_{i+2,h}}-u_{i+2,h}\|_0\|v_1\|_0
+\|u_{i+2,h}\|_0\|\NEW{\hat u_{i,h}}-u_{i,h}\|_0\\
&\le C h(\mu_{i+2}+\mu_i).
\endaligned
\label{eq:stimal3}
\end{equation}
Next, from the definition of $v_1$ and $v_2$ we have using
also~\eqref{eq:stimal2}
\begin{equation}
\label{eq:stimal4}
\aligned
|(\NEW{\hat u_{i+2,h}},v_2)|&\le|(\NEW{\hat u_{i+2,h}}-u_{i+2,h},v_2)|+
|(u_{i+2,h},\NEW{\hat u_{i+1,h}}-u_{i+1,h})|\\
&\quad +\frac{|(\NEW{\hat u_{i+1,h}},v_1)|}{\|v_1\|^2_0}
|(u_{i+2,h},\NEW{\hat u_{i,h}}-u_{i,h})|\\
&\le Ch(\mu_{i+2}+\mu_{i+1})+Ch^2(\mu_{i+1}+\mu_i)\mu_i.
\endaligned
\end{equation}
Therefore, putting things together, we get
\[
\|\nh(u_{i+2,h}-v_3)\|^2_0\le \mu_{i+2}^2
+Ch^2\sum_{l=i}^{i+1}(\mu_{i+2}+\mu_l)^2\lambda_{l,h}
+Ch^4(\mu_{i+1}+\mu_i)^2\mu_i^2\lambda_{i+1,h}.
\]
From this estimate it is straightforward to obtain the estimate of
$\|\nabla v_3\|_0$, and so on.
\end{proof}
It is now easy to obtain the desired a posteriori error estimate for the
eigenvalues approaching the continuous one from above.
\begin{thm}
\label{th:stimasopra}
Under the same hypotheses as in Lemma~\ref{le:main2} the following bound holds
true for $j=i,\dots,i+q-1$
\[
\frac{\lambda_{j,h}-\lambda_i}{\lambda_i}\le\frac1{\lambda_i}\left(
\frac1{J}\sum_{k=i}^{j}\left(6\mu_k^2+4C_1\eta_k^2+4|h.o.t.|^2\right)
+|h.o.t.|_2
+\frac1{J}\sum_{k=i}^{j-1}(\lambda_{j,h}-\lambda_{k,h})\right).
\]
\end{thm}
\begin{proof}
The proof is straightforward since
\[
\lambda_{j,h}-\lambda_i=
\frac1{J}\sum_{k=i}^{j-1}\left(\lambda_{j,h}-\lambda_{k,h}\right)
+\frac1{J}\sum_{k=i}^j\lambda_{k,h}-\lambda_i,
\]
which combined with~\eqref{eq:stimamedia} gives the result.

\end{proof}

\section{Numerical results}

In this section we present some numerical results obtained applying the
estimator discussed in this paper.

The aim of our tests is not to show the reliability and the efficiency of
the error indicator, since results in this direction have been already
reported in~\cite{DDP}. The emphasis of our computations will be put on the
approximation of multiple eigenvalues. There are few papers concerning
adaptivity for the numerical approximation of multiple eigenvalues, among
those we recall in particular~\cite{SoGi},~\cite{zhou}, and~\cite{Gallistl}.
A consequence
of the results contained in these \NEW{three} papers is, in particular, that an
adaptive algorithm aiming at a refinement procedure for the approximation of a
multiple eigenvalue, needs to take into account indicators belonging to all
the discrete eigenfunctions approximating the continuous eigenspace.

In all our numerical tests the discrete eigenvalues are lower bounds for the
continuous ones.
\NEW{Actually, we checked the first one hundred eigenvalues of our numerical
tests and all of them appear to be lower bounds for the corresponding exact
frequencies. Our experience, is that the method generally provides eigenvalues
converging from below (apart from very pathological examples with few
elements). So far, a proof of this fact is still missing.}

\NEW{We} focus our attention on eigenvalues of multiplicity
two.
We consider an adaptive algorithm based on standard D\"orfler marking
strategy~\cite{Willy}. In case of double eigenvalues, we choose either to
mark elements according to the error indicator based on one of the two
discrete corresponding eigenfunctions, or to the sum of the indicators of the
two eigenfunctions.
From the discussion of the results we are going to present, it will be clear
that the results of~\cite{SoGi},~\cite{zhou}, and~\cite{Gallistl} are
confirmed. A more comprehensive and detailed study of the numerical tests
will be included in a forthcoming paper.

%
%
%

\subsection{The square ring domain}

Let $\Omega$ be the domain obtaining by subtracting the square $(1/3,2/3)^2$
from the square $(0,1)^2$. This is the same domain considered
in~\cite[Sect.~2]{SoGi}.

It turns out that $\lambda_2$ is an eigenvalue with multiplicity two. More
precisely, the two-dimensional eigenspace corresponding to
$\lambda_2=\lambda_3$ can be generated by two singular
eigenfunctions: each of
them has a singularity at two opposite reentrant corners of $\Omega$. \NEW{We
used the reference value $\lambda_2=\lambda_3=84.517\dots$ computed by Aitken
extrapolation using very fine meshes and conforming linear elements.}

\NEW{We use Matlab for our tests, so that the eigenvalues/eigenvectors
are computed with the command \texttt{eigs}. At each stage we compute the
three smallest eigenvalues and take into account the second and third ones.}

We start by using a refinement strategy based on the discrete eigenfunction
$u_{2,h}$ (notice that, due to the fact that the discrete eigenvalues are
lower bounds, this corresponds to the \NEW{one which is farther away from the
double eigenvalue we want to approximate}).
The initial non-structured mesh is shown in Figure~\ref{fg:ring_irre_mesh1}.
\begin{figure}
\includegraphics[width=8cm]{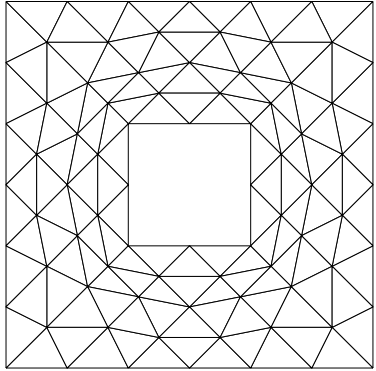}
\caption{The initial mesh for the square ring test case}
\label{fg:ring_irre_mesh1}
\end{figure}
The plot of the discrete eigenvalues is reported in
Figure~\ref{fg:ring_irre_eig}.

\begin{figure}
\includegraphics[width=8cm]{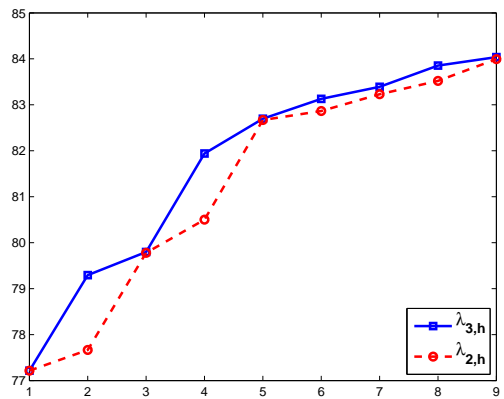}
\caption{The values of the discrete eigenvalues $\lambda_{2,h}$ and
$\lambda_{3,h}$ on the square ring for different refinement levels (non
structured initial mesh, refined based on $u_{2,h}$)}
\label{fg:ring_irre_eig}
\end{figure}

It is clear that the phenomenon already explained in~\cite{SoGi} is present.
The two discrete eigenvalues apparently change their position as the mesh is
refined. For example, on the first mesh the two values are equal since the
mesh is symmetric.
\NEW{It turns out that the solver determines a decomposition of the discrete
eigenspace: let's call $u_{j,h_1}$ ($j=2,3$) the two eigenfunctions on the
initial mesh. The first
refinement step is then performed by using the
eigenfunction $u_{2,h_1}$; it happens that on the second mesh $\lambda_{2,h_2}$ and
$\lambda_{3,h_2}$ are separated and correspond to new eigenfunctions
$u_{2,h_2}$ and $u_{3,h_2}$. Actually, due to the performed refinement, this
can be interpreted as if $u_{3,h_2}$ was an improved approximation of
$u_{2,h_1}$; in some sense, the position of the two eigenfunctions has exchanged
after the first refinement step. This phenomenon is more clear by looking at
the corresponding meshes (see Figure~\ref{fg:ring_irre_fun}): the second
mesh has been refined at the North-West and South-East corners, while the new
second eigenfunction $u_{2,h_2}$ is singular at the other two opposite corners.
}

\begin{figure}
\includegraphics[width=6cm]{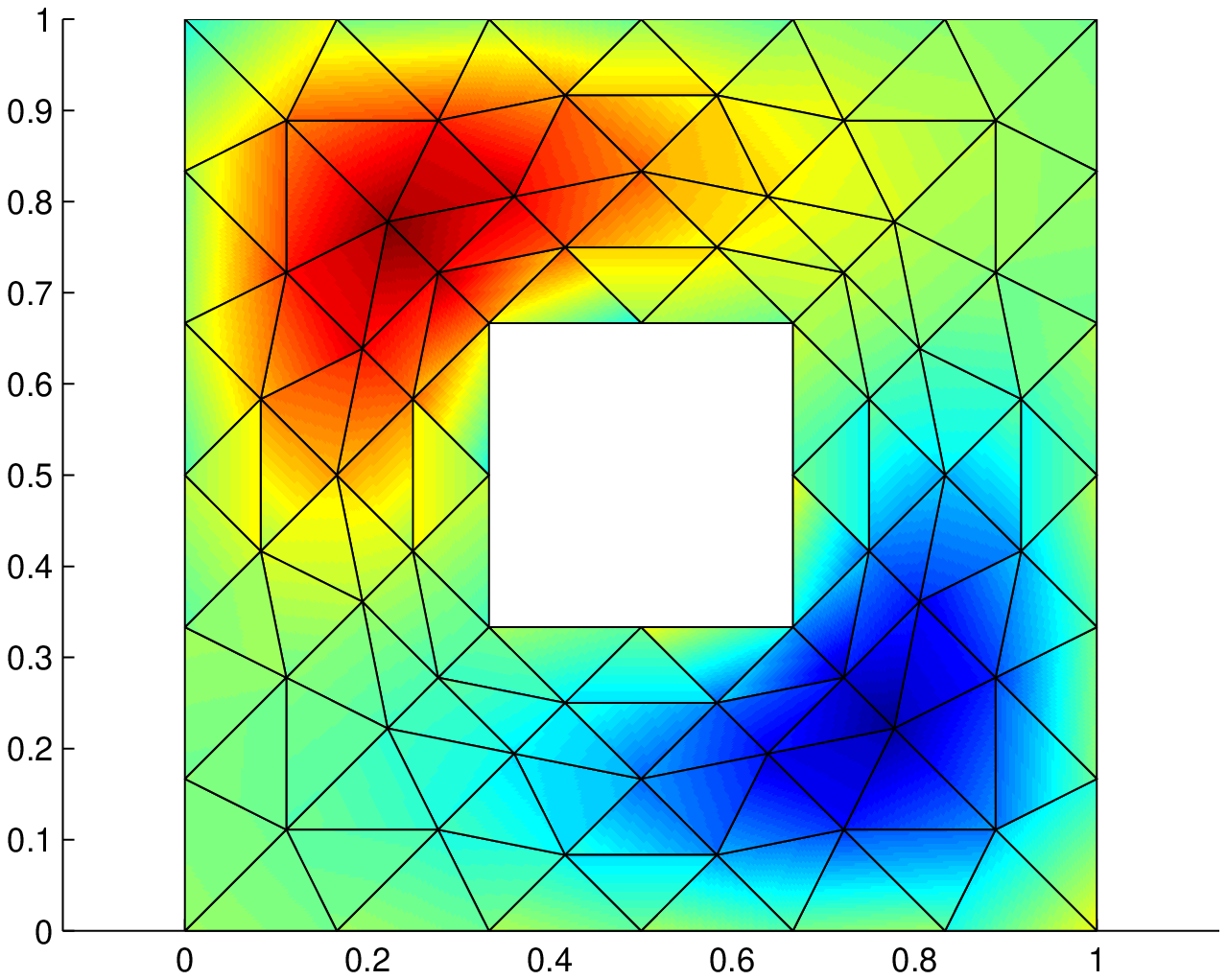}
\includegraphics[width=6cm]{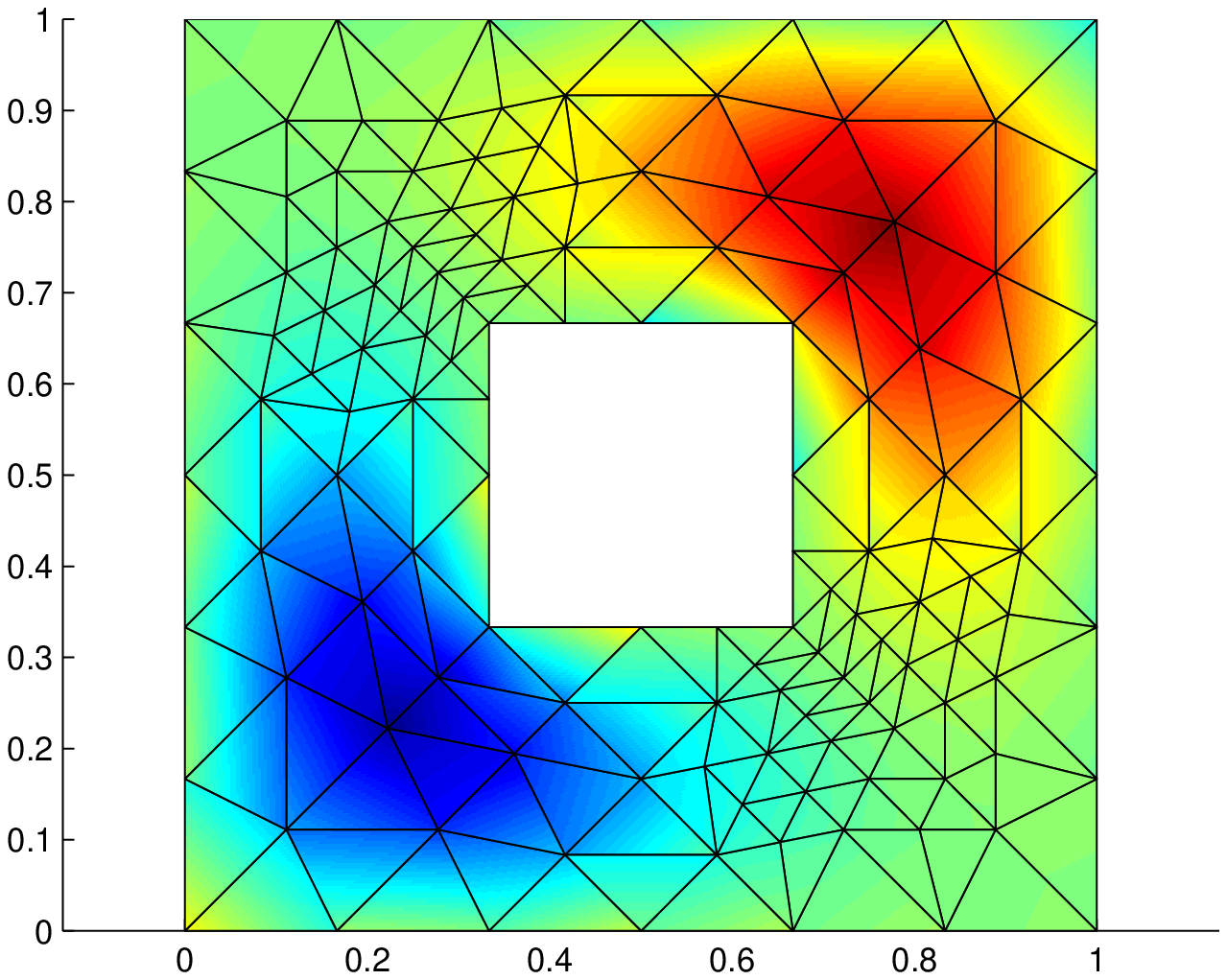}

\includegraphics[width=6cm]{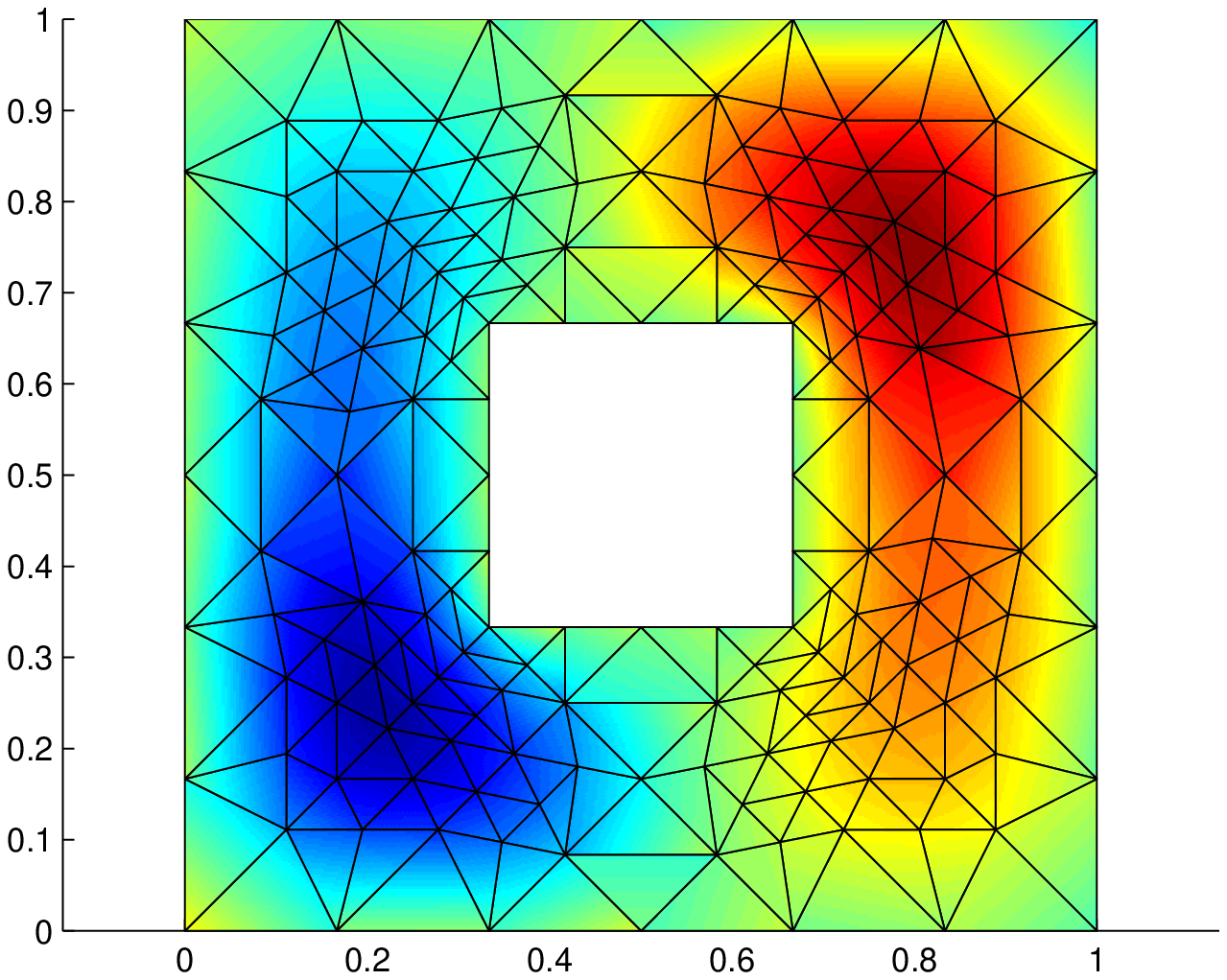}
\includegraphics[width=6cm]{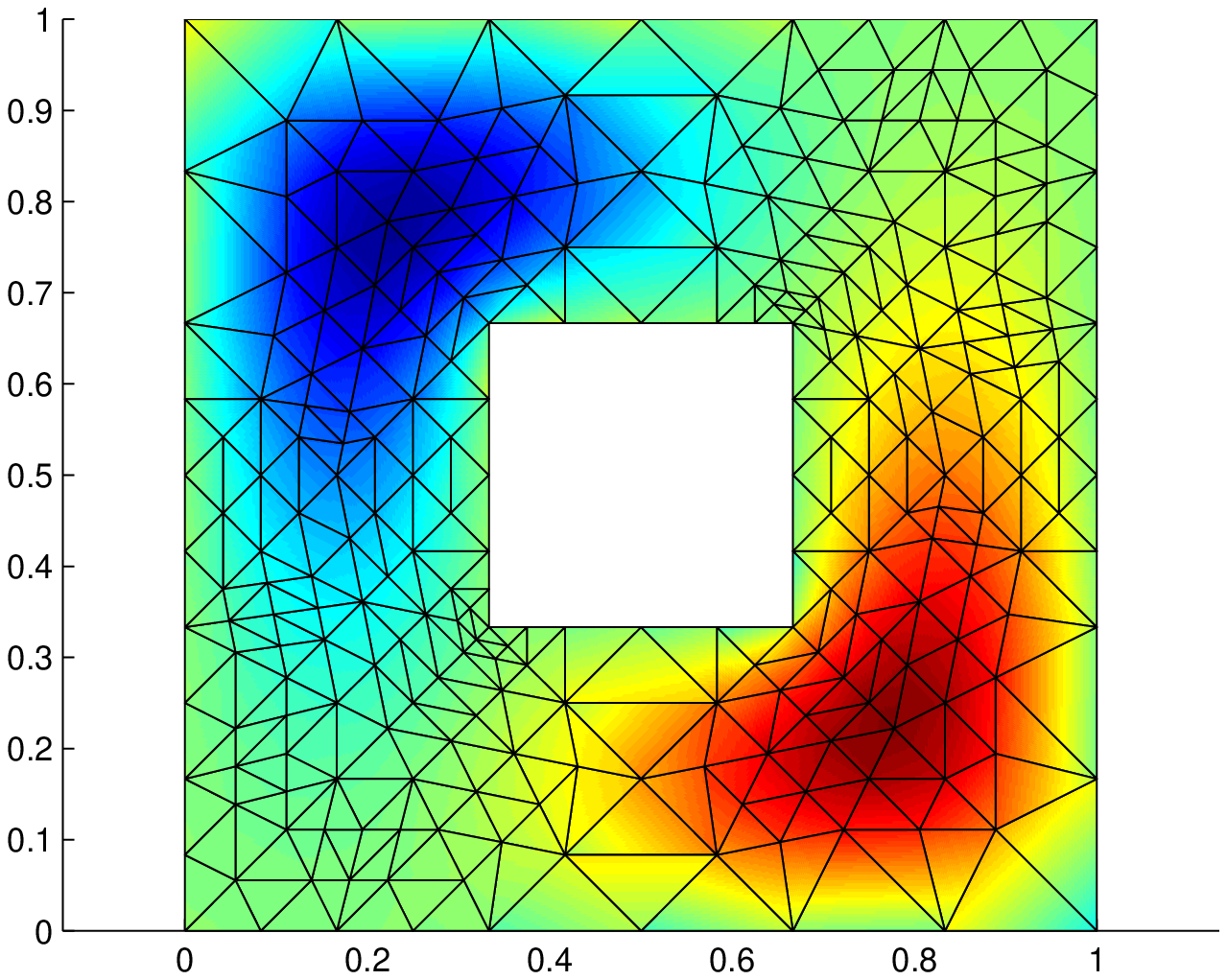}
\caption{The eigenfunctions corresponding to $\lambda_{2,h}$ on the square
ring for the first four refinement levels (non structured initial mesh, refined
based on $u_{2,h}$)}
\label{fg:ring_irre_fun}
\end{figure}

In order to better highlight this phenomenon, in Figure~\ref{fg:ring-allmesh}
we report a sequence of meshes obtained after eight level of refinements. It
is clear that the method refines a region close to all four reentrant corners
even if we are constructing our indicator only according to the second
discrete eigenfunction. Moreover, the refinement strategy is not optimal since
at each step only two of the four involved regions are considered.

\begin{figure}
\includegraphics[width=6cm]{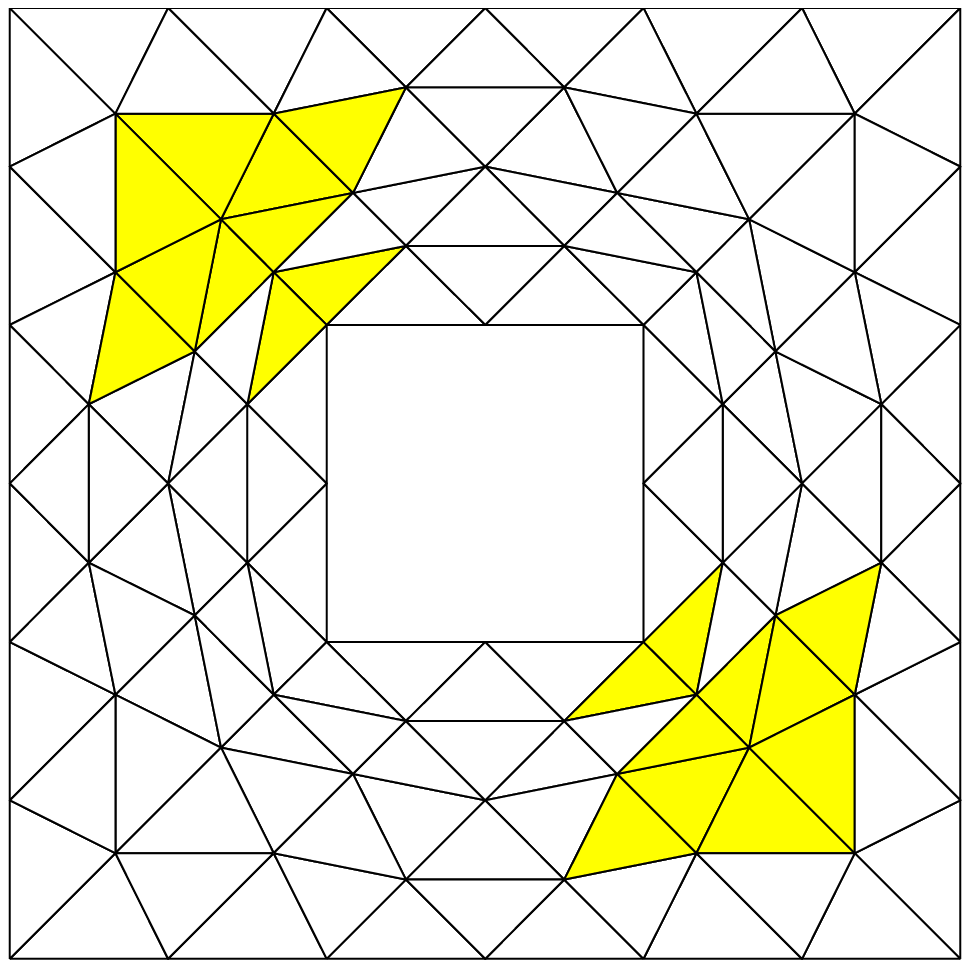}
\includegraphics[width=6cm]{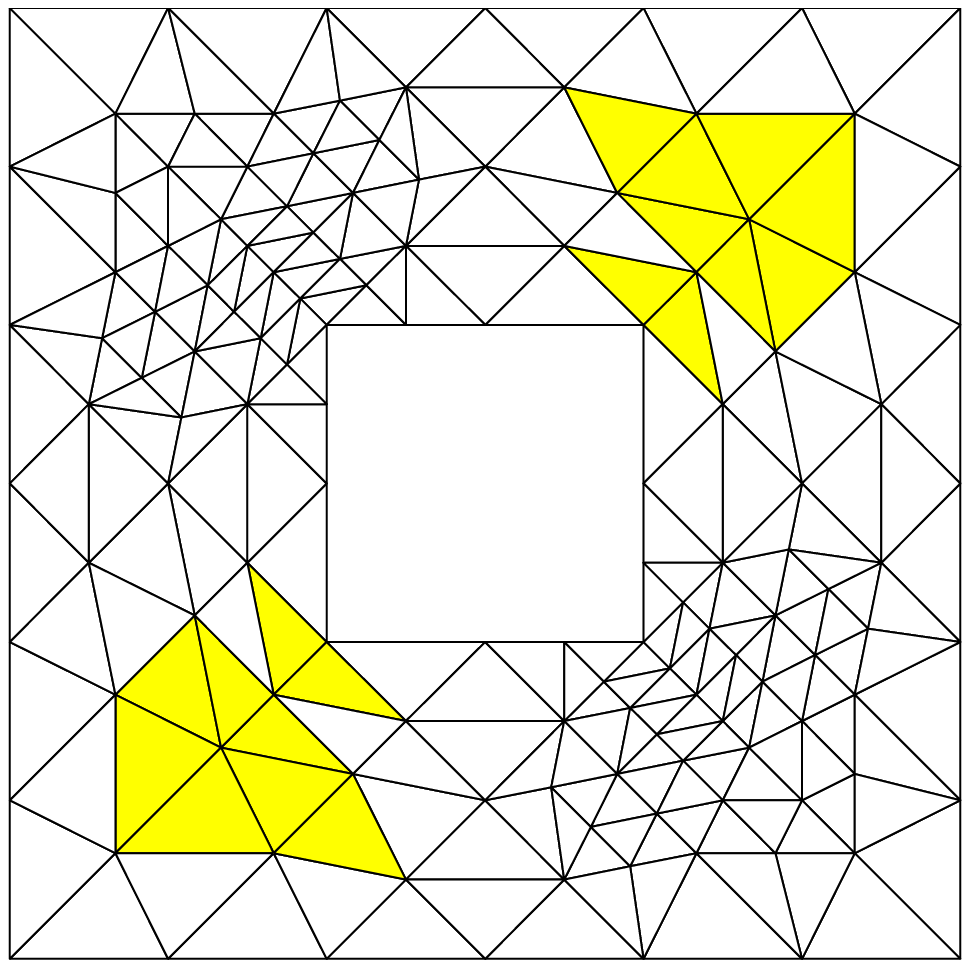}

\includegraphics[width=6cm]{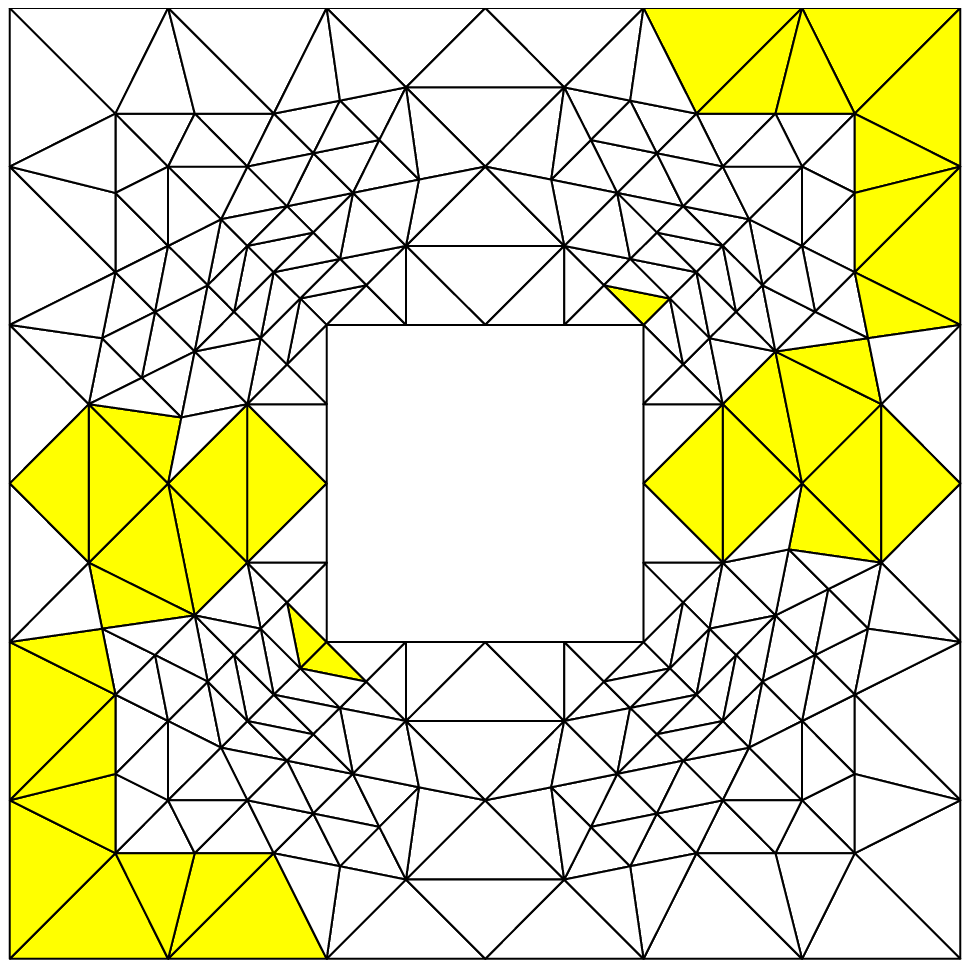}
\includegraphics[width=6cm]{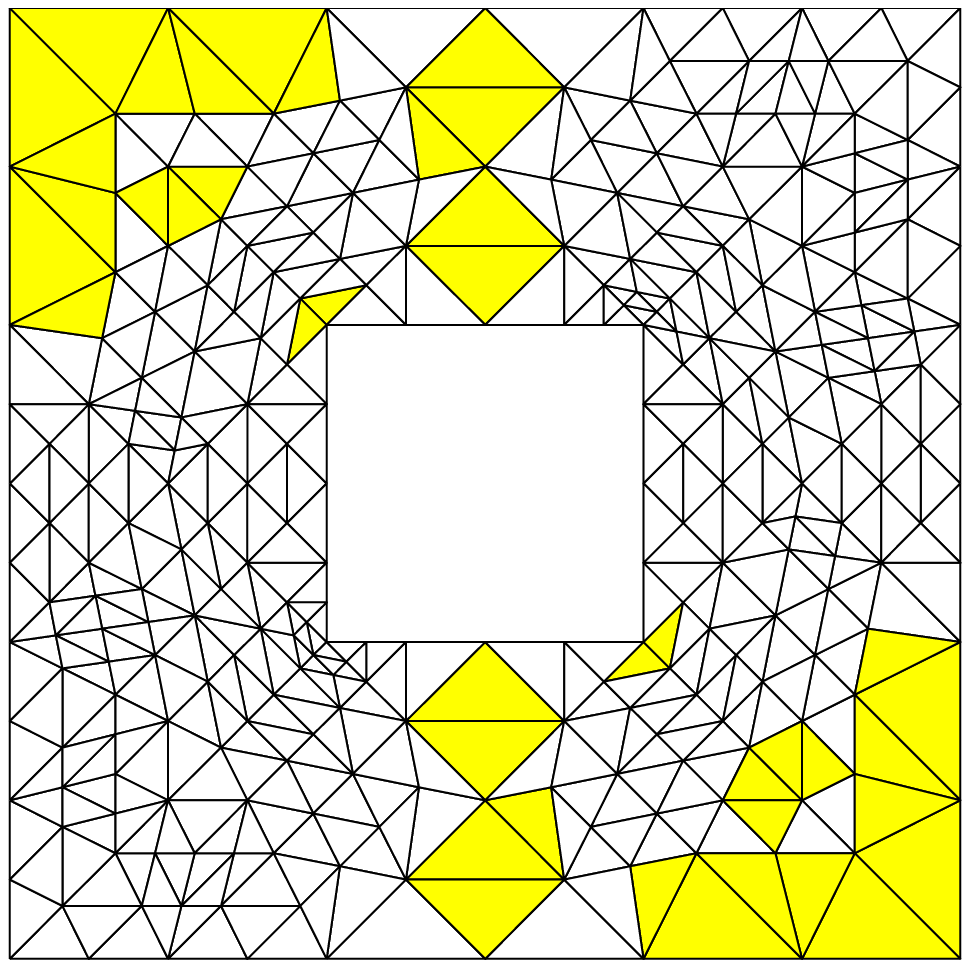}

\includegraphics[width=6cm]{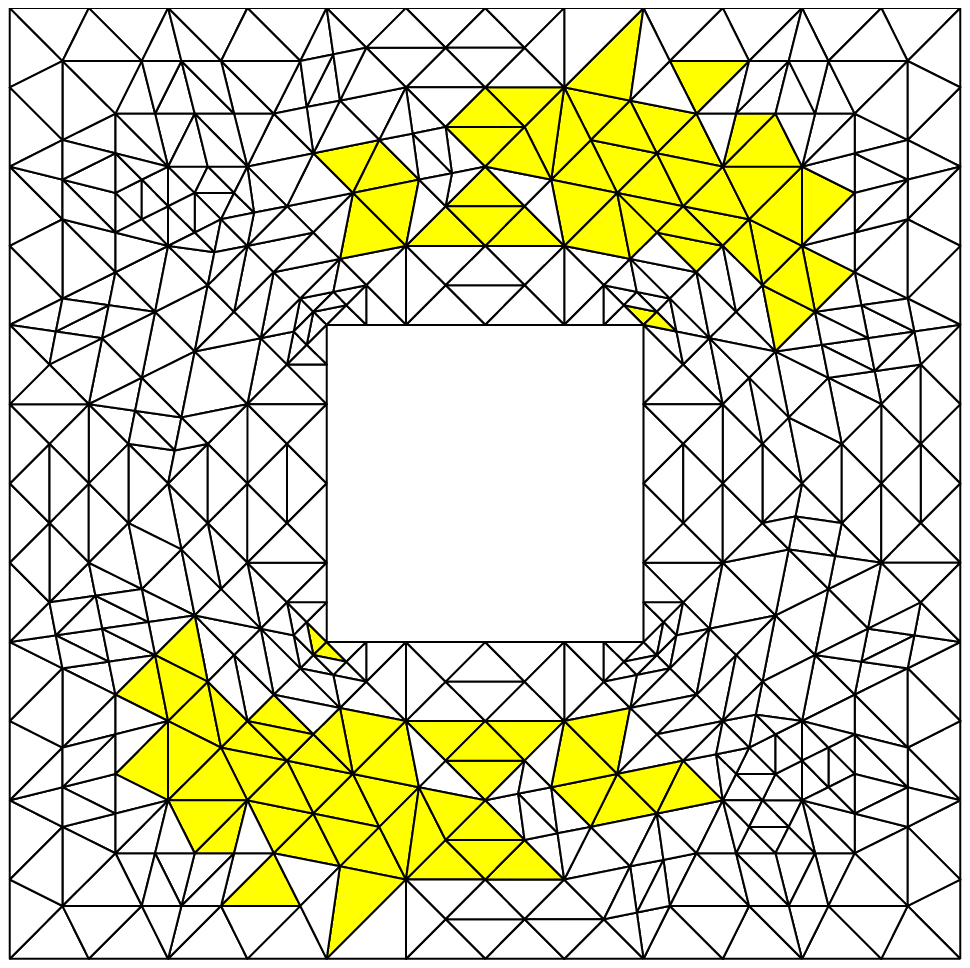}
\includegraphics[width=6cm]{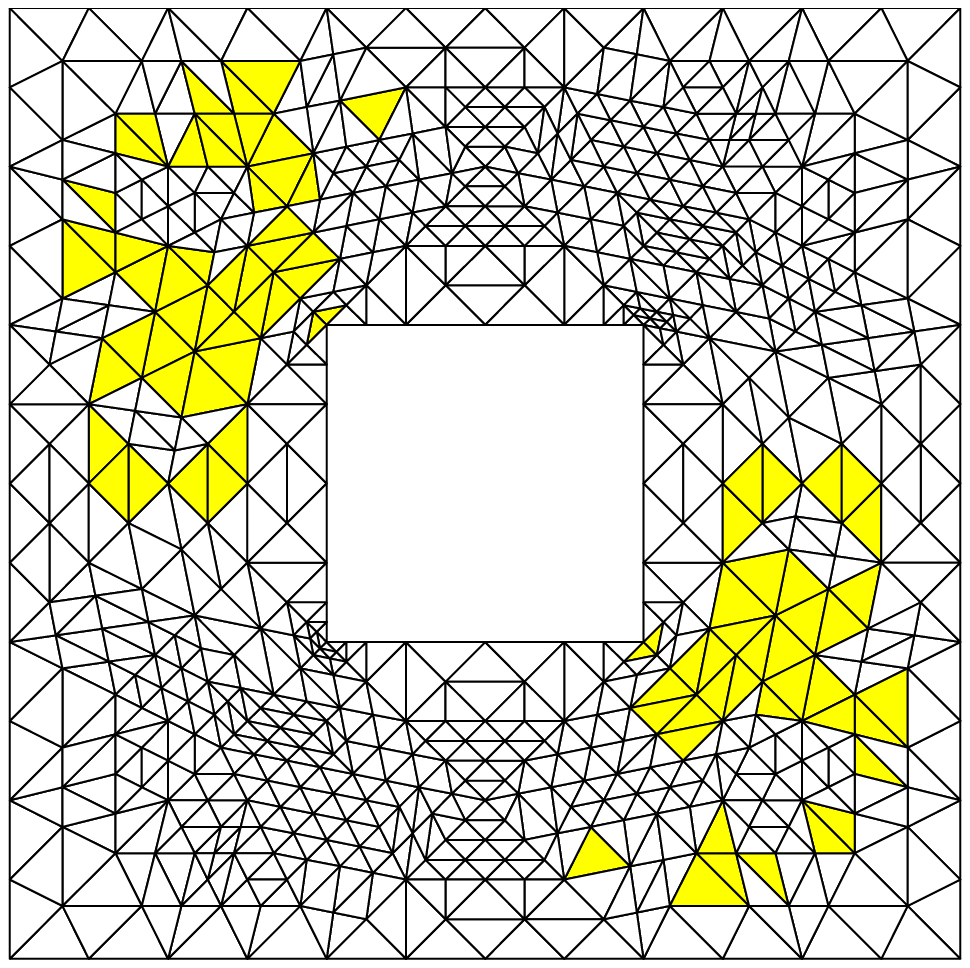}

\includegraphics[width=6cm]{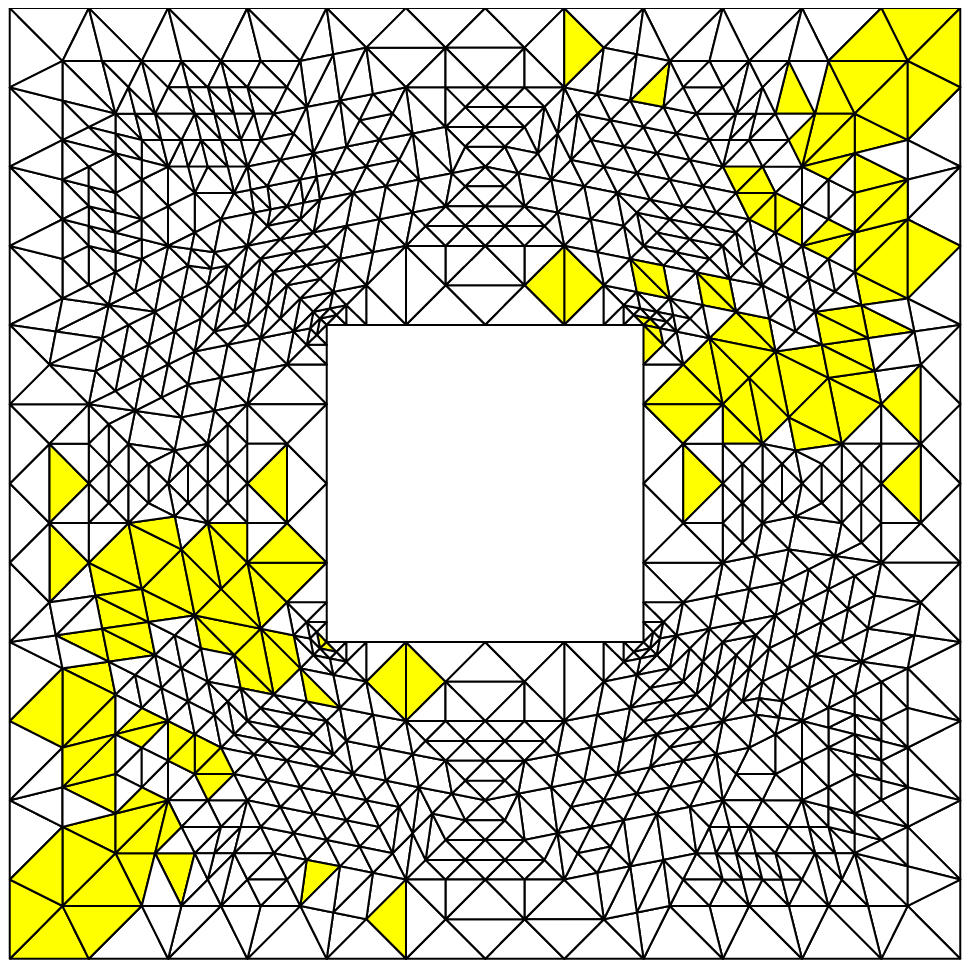}
\includegraphics[width=6cm]{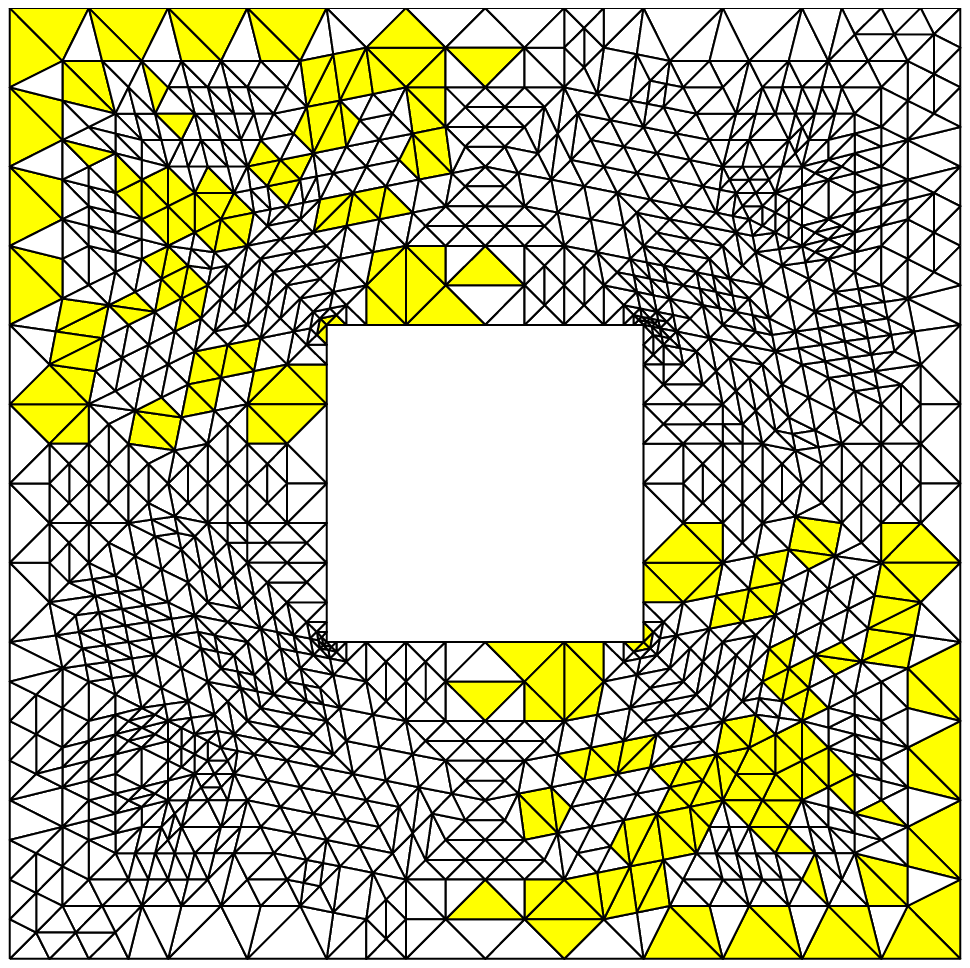}
\caption{Eight level of refinements of the square ring domain: indicator based
on $u_{2,h}$}
\label{fg:ring-allmesh}
\end{figure}

As a second test, we consider an adaptive strategy driven by the approximation
of the eigenfunction $u_{3,h}$ corresponding to the third eigenvalue. A plot
of the values of $\lambda_{2,h}$ and $\lambda_{3,h}$ is reported in
Figure~\ref{fg:ring_irre_eig1}.
\begin{figure}
\includegraphics[width=8cm]{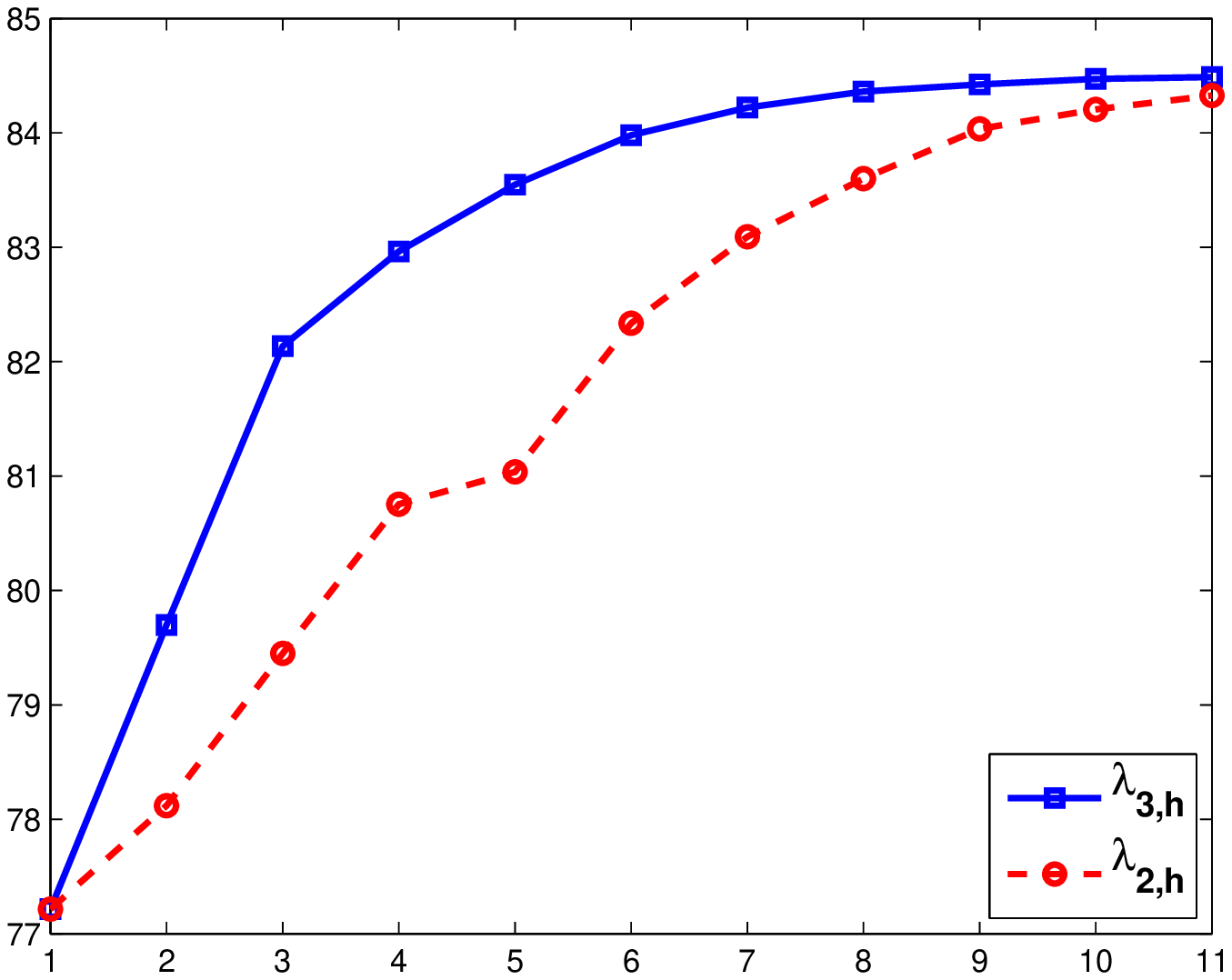}
\caption{The values of the discrete eigenvalues $\lambda_{2,h}$ and
$\lambda_{3,h}$ on the square ring for different refinement levels (non
structured initial mesh, refined based on $u_{3,h}$)}
\label{fg:ring_irre_eig1}
\end{figure}
It is clear that in this case the adaptive procedure is effective in pushing
the convergence of $\lambda_{3,h}$, while $\lambda_{2,h}$ is converging more
slowly.

The corresponding meshes are reported in Figure~\ref{fg:ring-allmesh1}, where
it can be seen that the refinements is always performed in a neighborhood of
the same two corners.

\begin{figure}
\includegraphics[width=6cm]{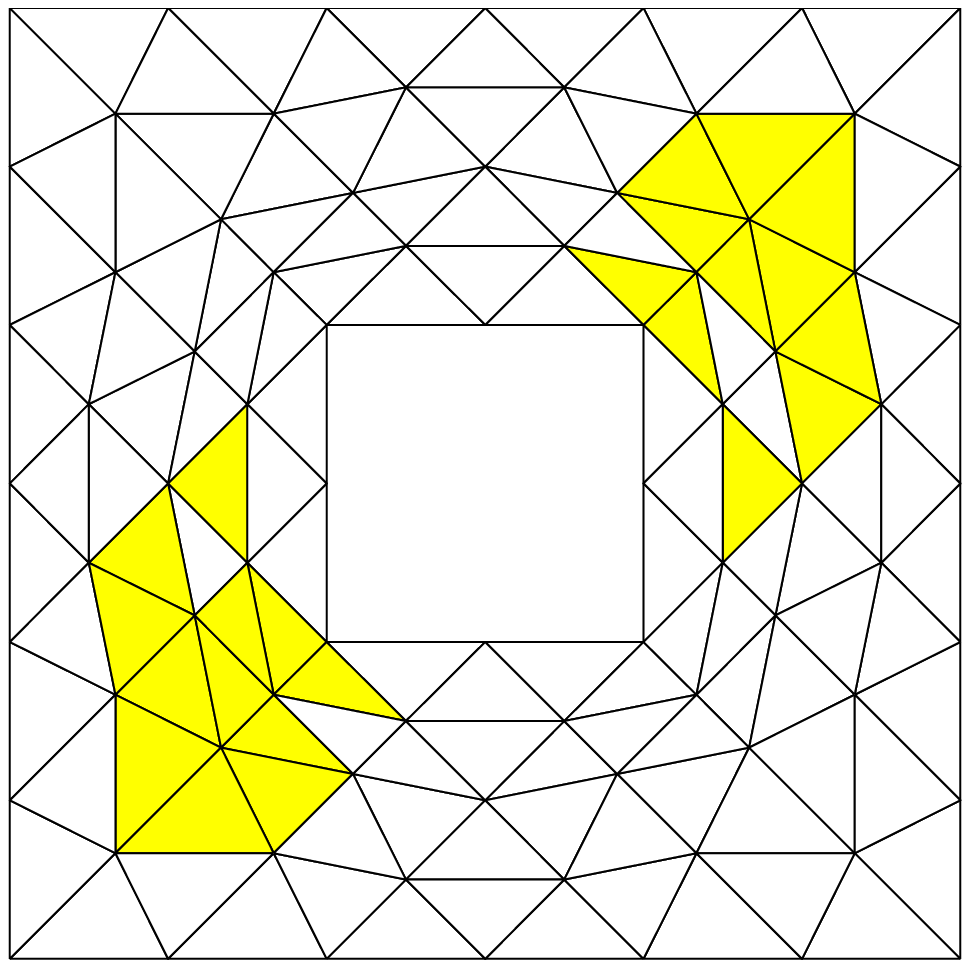}
\includegraphics[width=6cm]{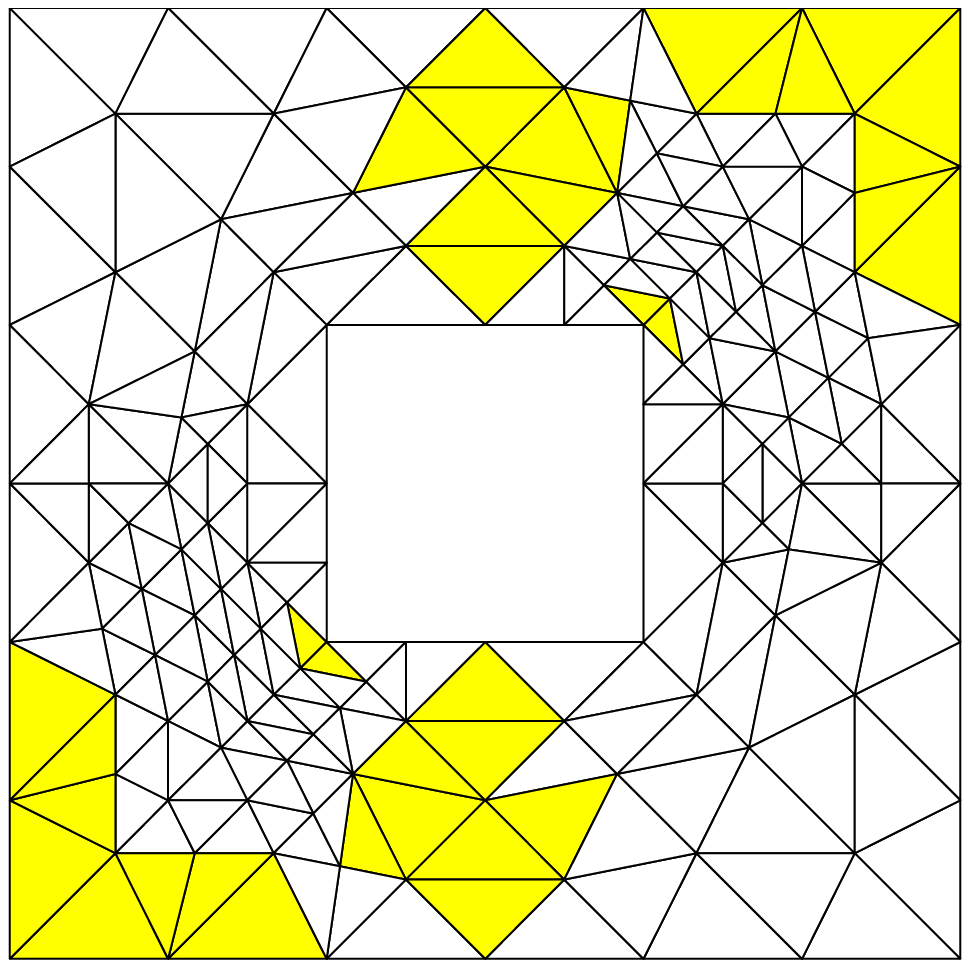}

\includegraphics[width=6cm]{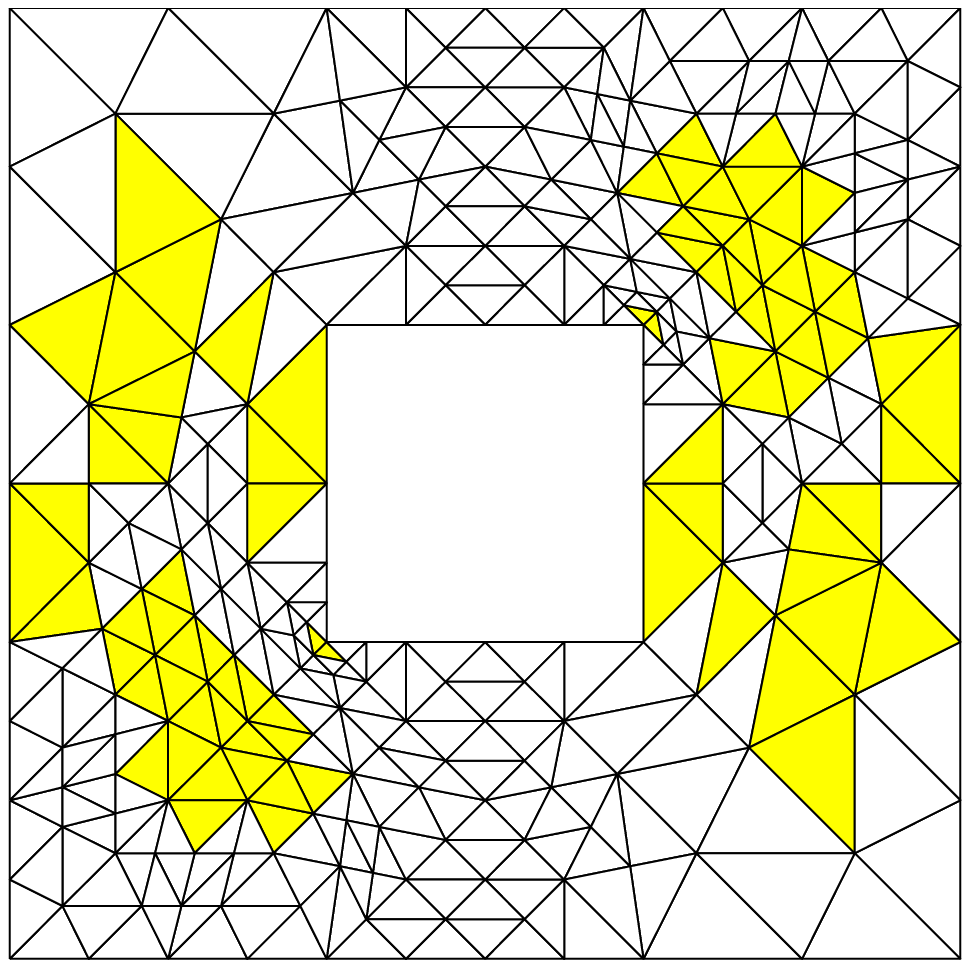}
\includegraphics[width=6cm]{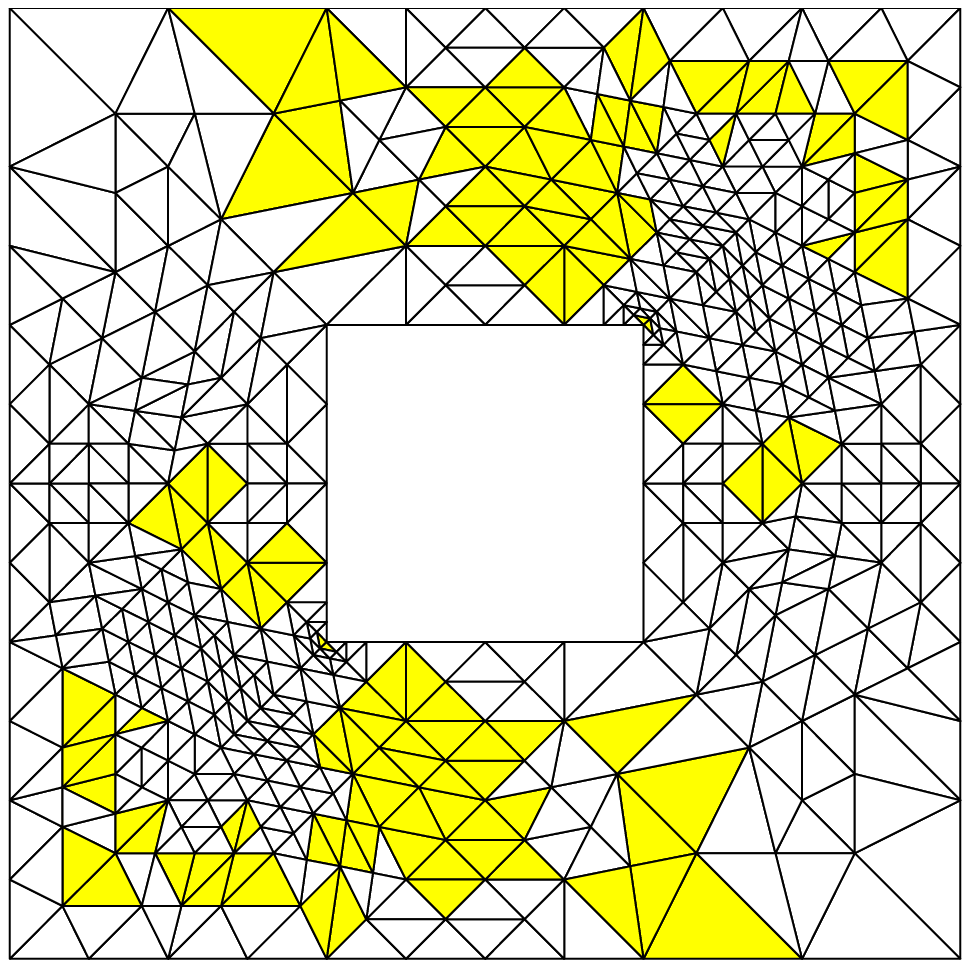}

\includegraphics[width=6cm]{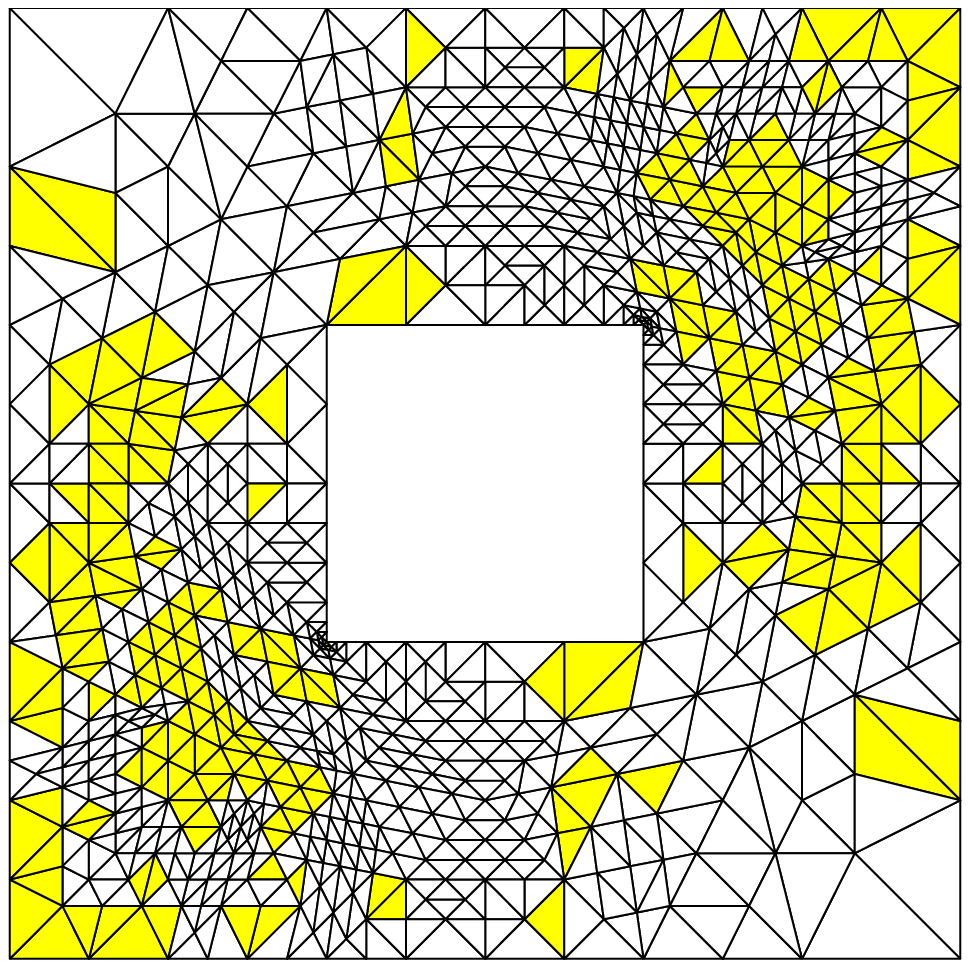}
\includegraphics[width=6cm]{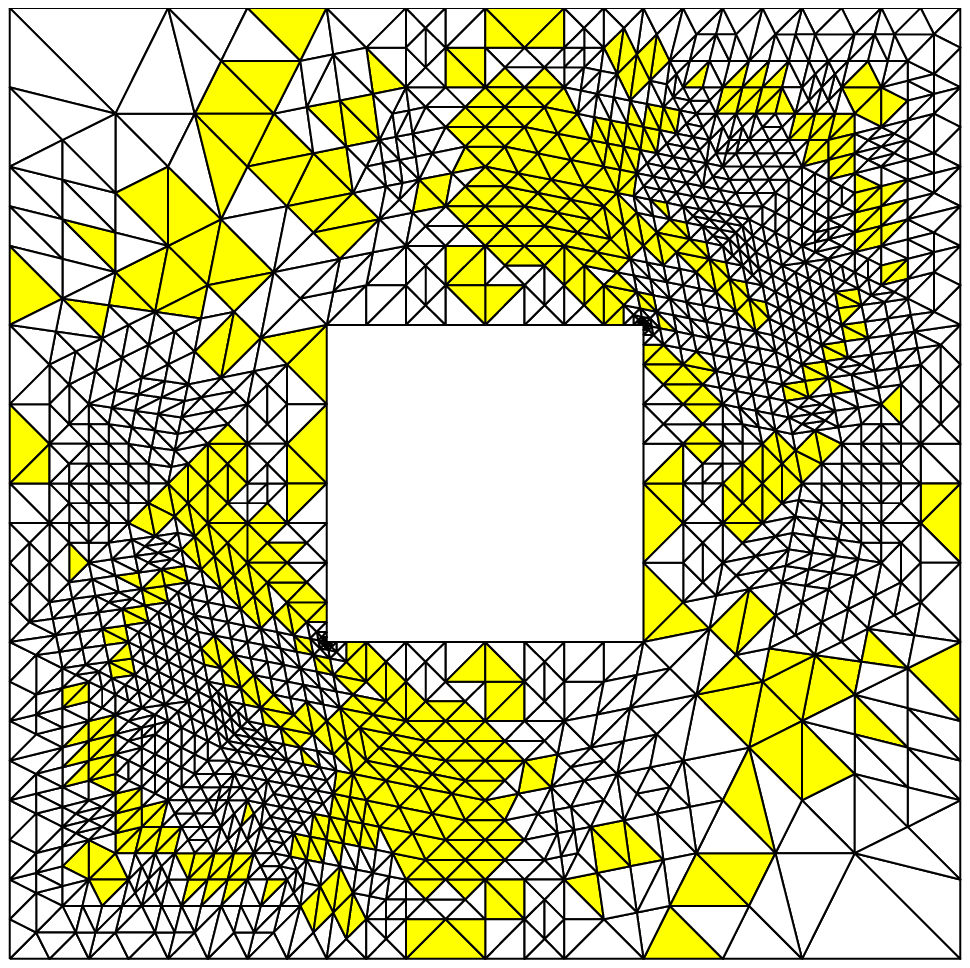}

\includegraphics[width=6cm]{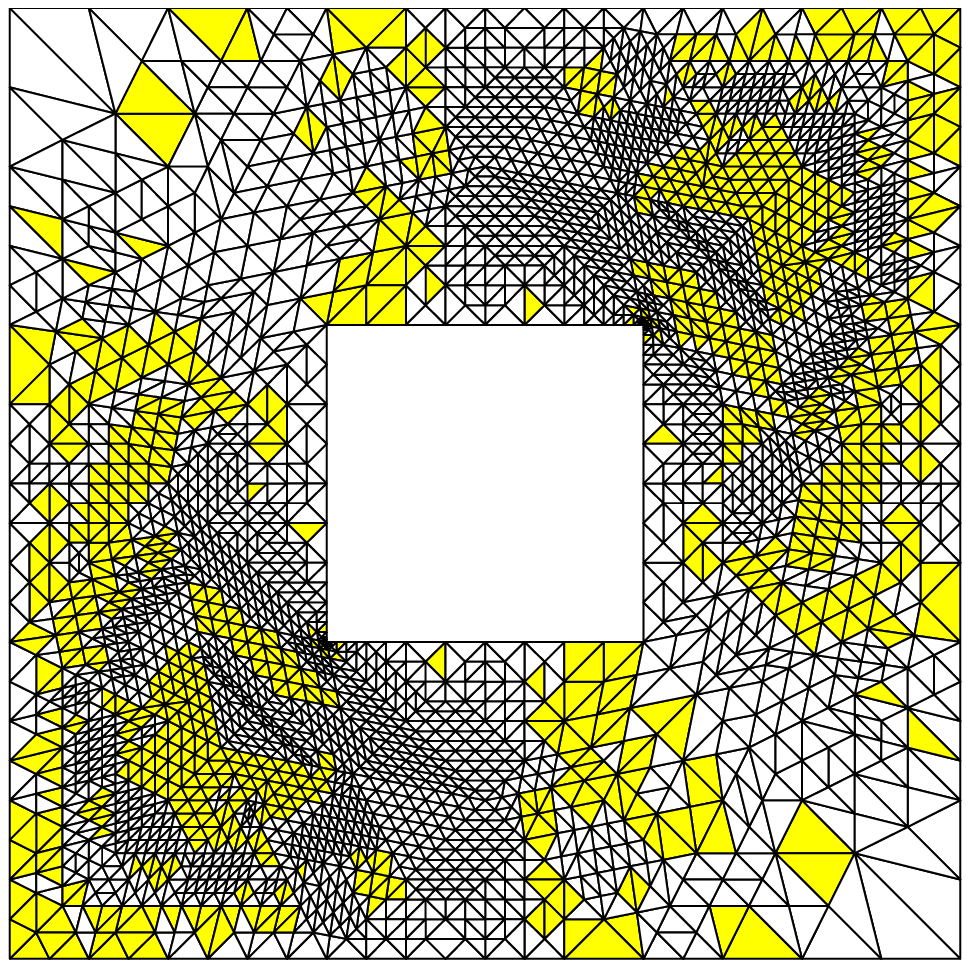}
\includegraphics[width=6cm]{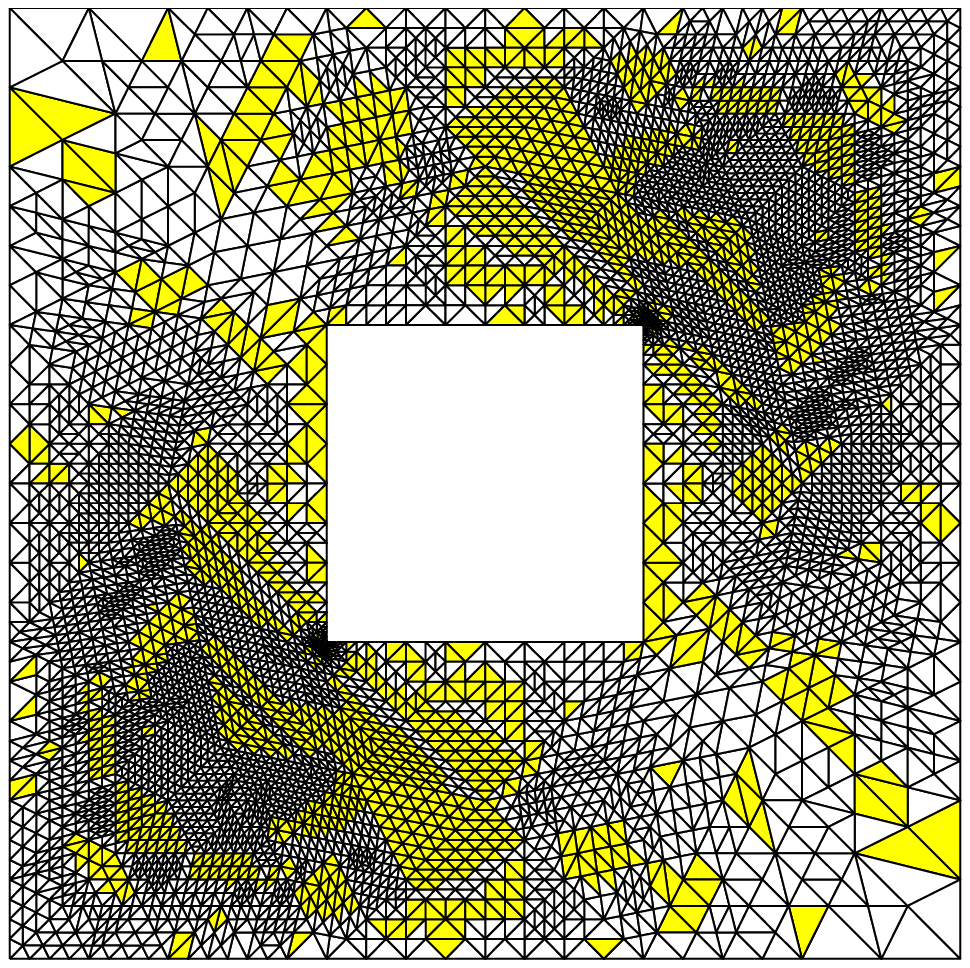}
\caption{Eight level of refinements of the square ring domain: indicator based
on $u_{3,h}$}
\label{fg:ring-allmesh1}
\end{figure}

The eigenfunctions on the first four meshes are plotted in
Figure~\ref{fg:ring_irre_fun1}: it can be seen that they always have
singularities about the top right and bottom left reentrant corners.
\NEW{It should be clear that this behavior is just a good luck of this
particular situation. Readers are warned that in case of invariant spaces of
dimension greater than one, an effective adaptive strategy should consider
all involved discrete eigenfunctions}.

\begin{figure}
\includegraphics[width=6cm]{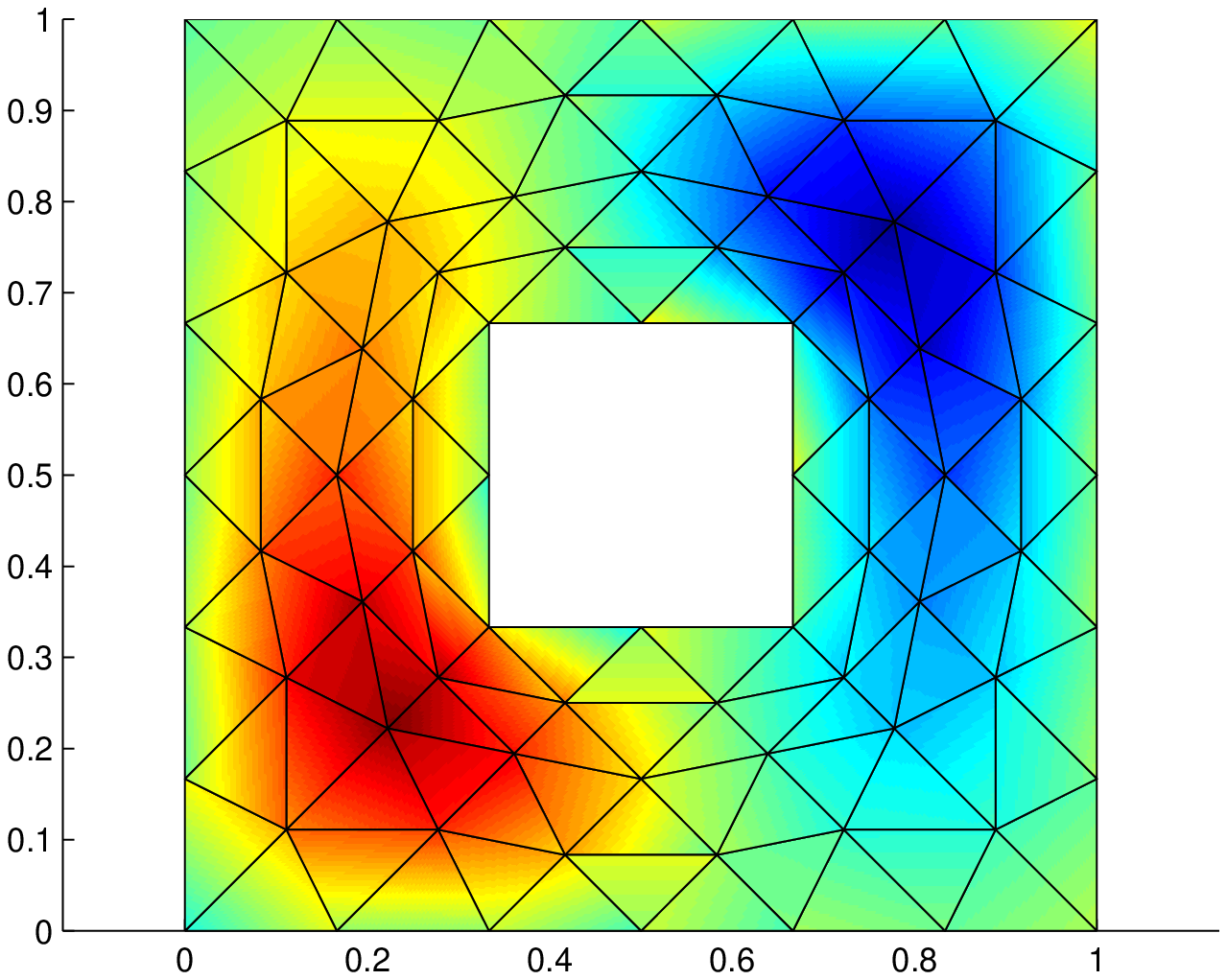}
\includegraphics[width=6cm]{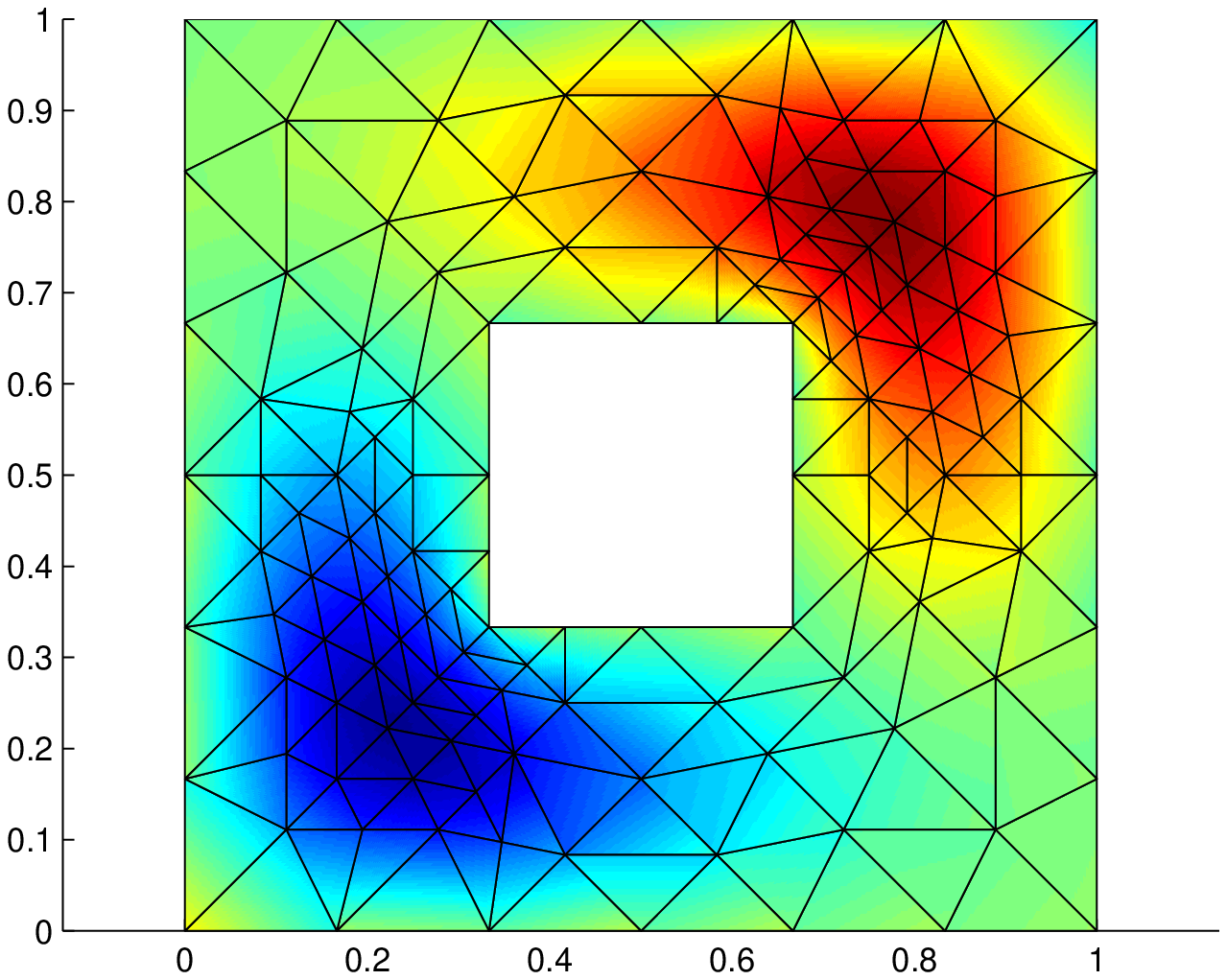}

\includegraphics[width=6cm]{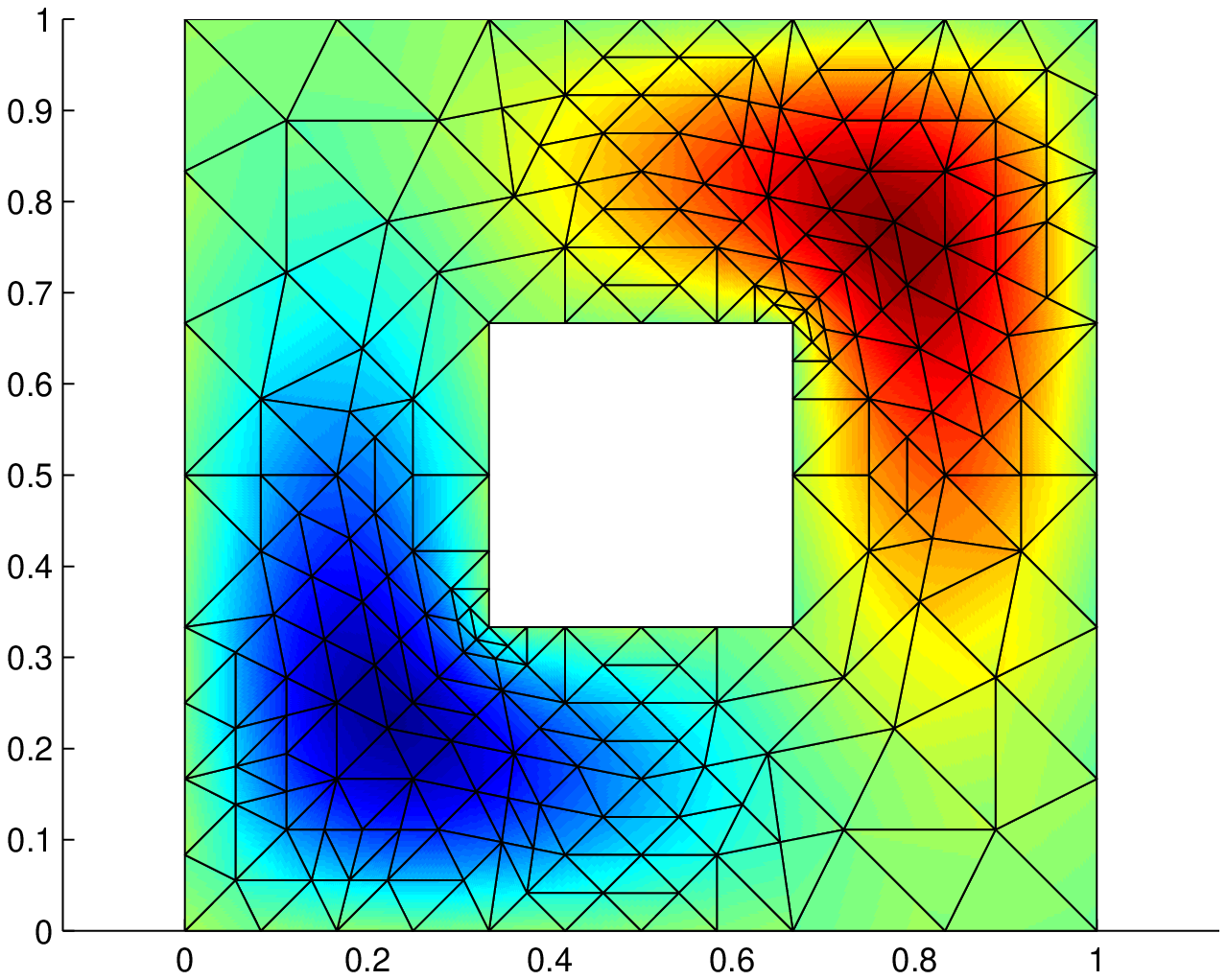}
\includegraphics[width=6cm]{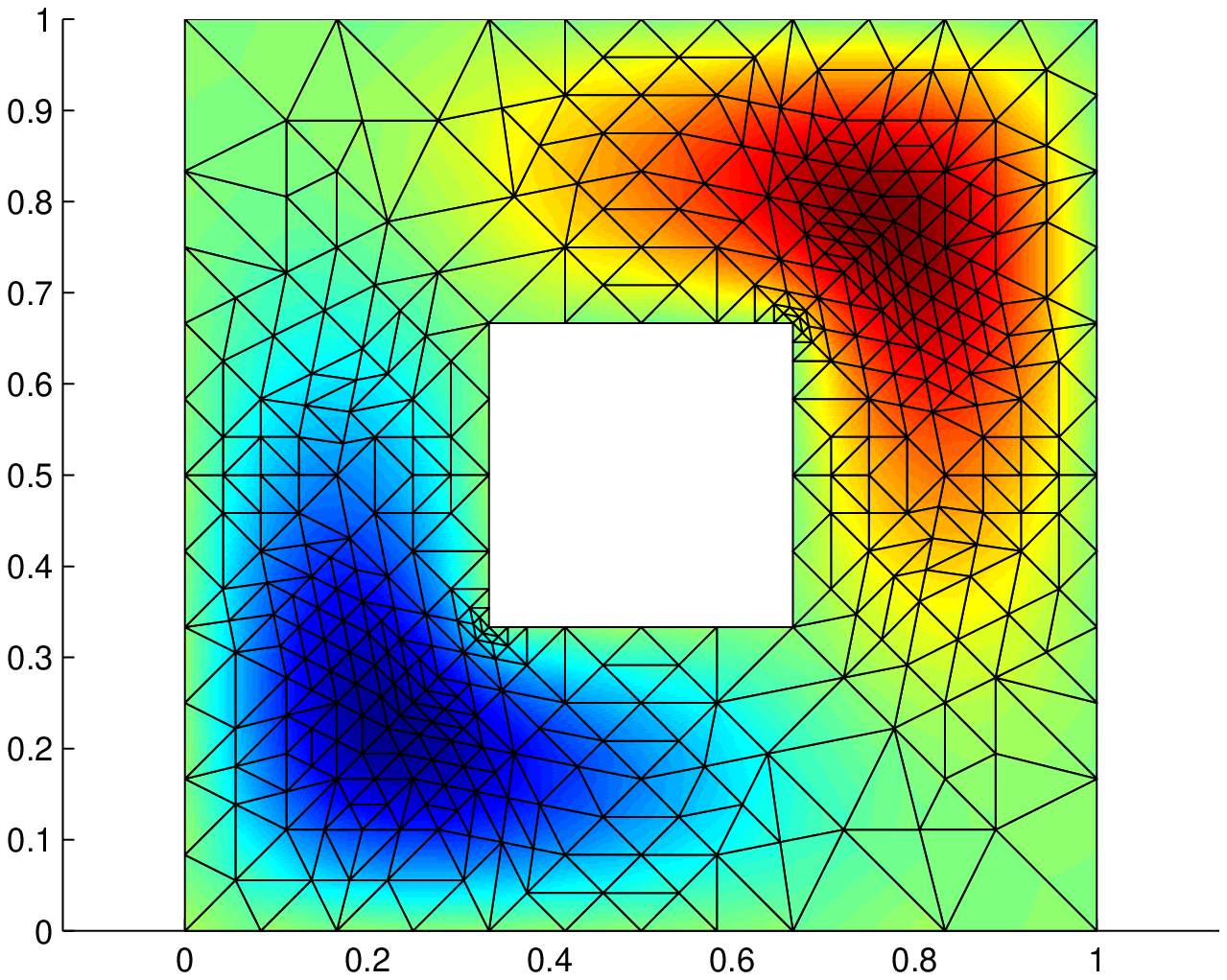}
\caption{The eigenfunctions corresponding to $\lambda_{3,h}$ on the square
ring for the first four refinement levels (non structured initial mesh,
refined based on $u_{3,h}$)}
\label{fg:ring_irre_fun1}
\end{figure}

\NEW{This approach is presented as a conclusion of this section, where we
repeat} the same computation with an error
indicator based on both singular eigenfunctions $u_{2,h}$ and $u_{3,h}$.
The plot of the eigenvalues is reported in Figure~\ref{fg:ring_irre_eig2}
where it can be observed that now the two discrete values approximating the
double eigenvalue $\lambda_2=\lambda_3$ are almost superimposed.

\begin{figure}
\includegraphics[width=8cm]{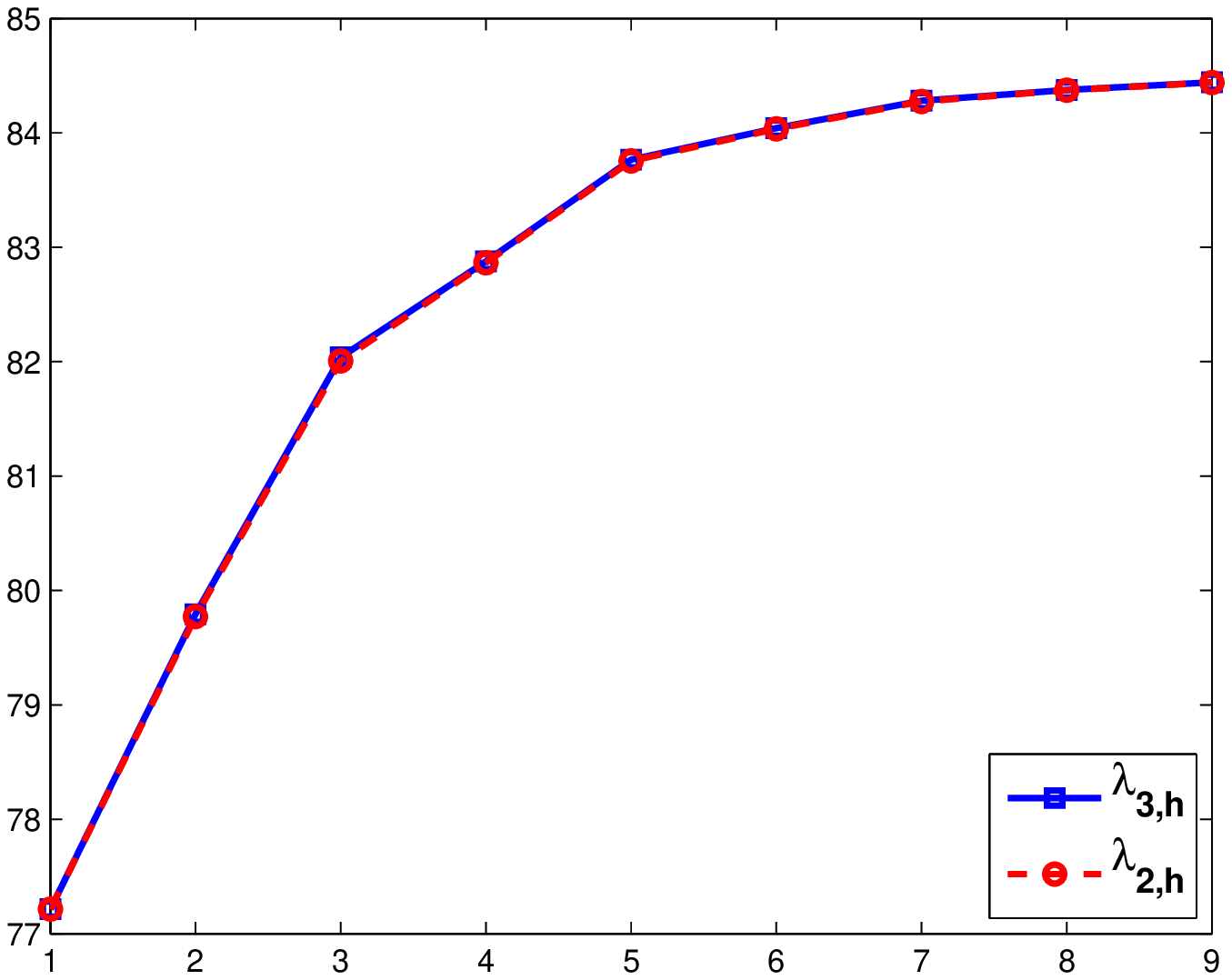}
\caption{The values of the discrete eigenvalues $\lambda_{2,h}$ and
$\lambda_{3,h}$ on the square ring for different refinement levels (non
structured initial mesh, refined based on both $u_{2,h}$ and $u_{3,h}$)}
\label{fg:ring_irre_eig2}
\end{figure}

For completeness, we report in Figure~\ref{fg:ring_irre_fun2} the
eigenfunctions corresponding to $\lambda_{2,h}$ and $\lambda_{3,h}$ after
three refinements and in Figure~\ref{fg:ring-allmesh2} the sequence of
mesh obtained after eight level of refinements.

\begin{figure}
\includegraphics[width=6cm]{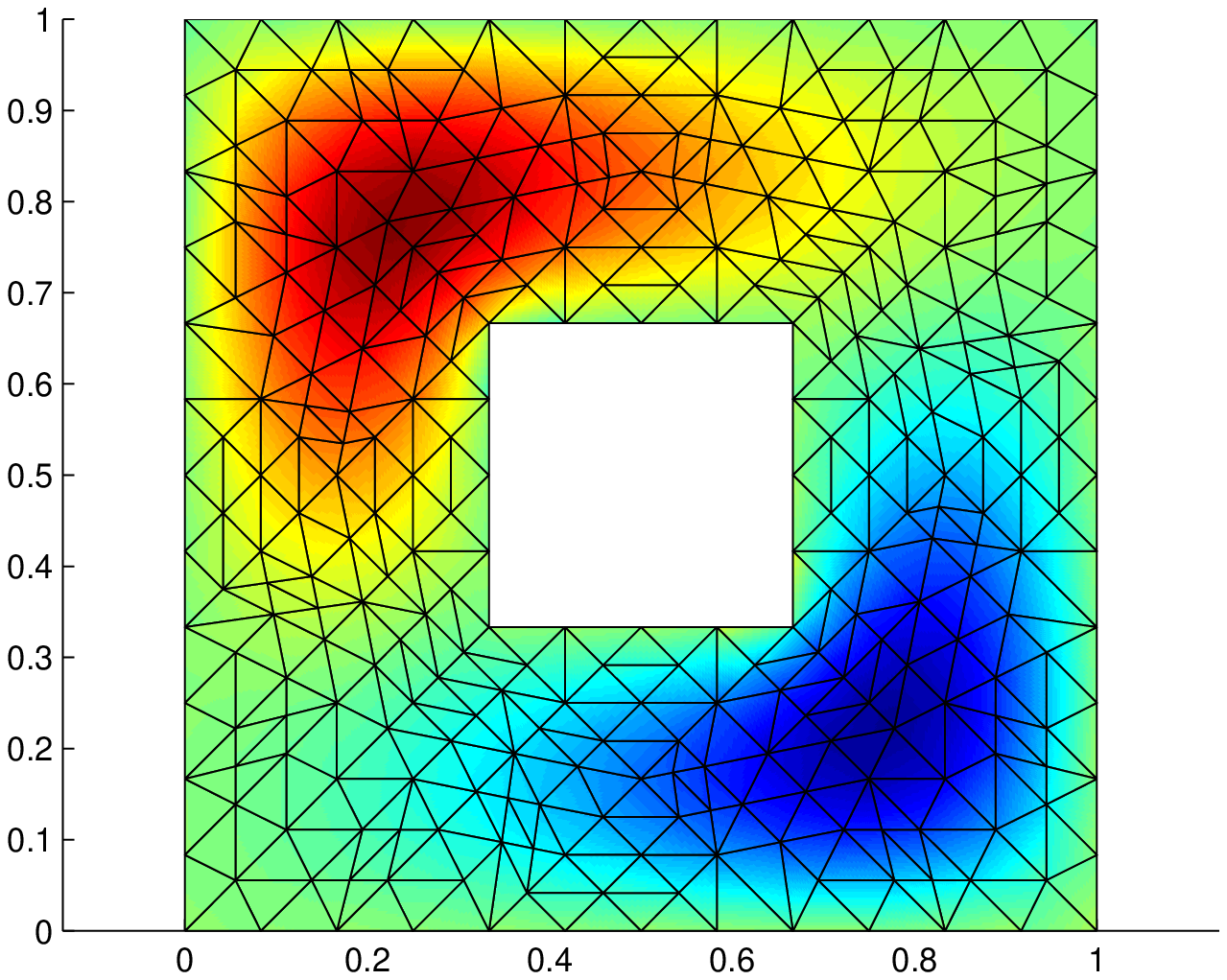}
\includegraphics[width=6cm]{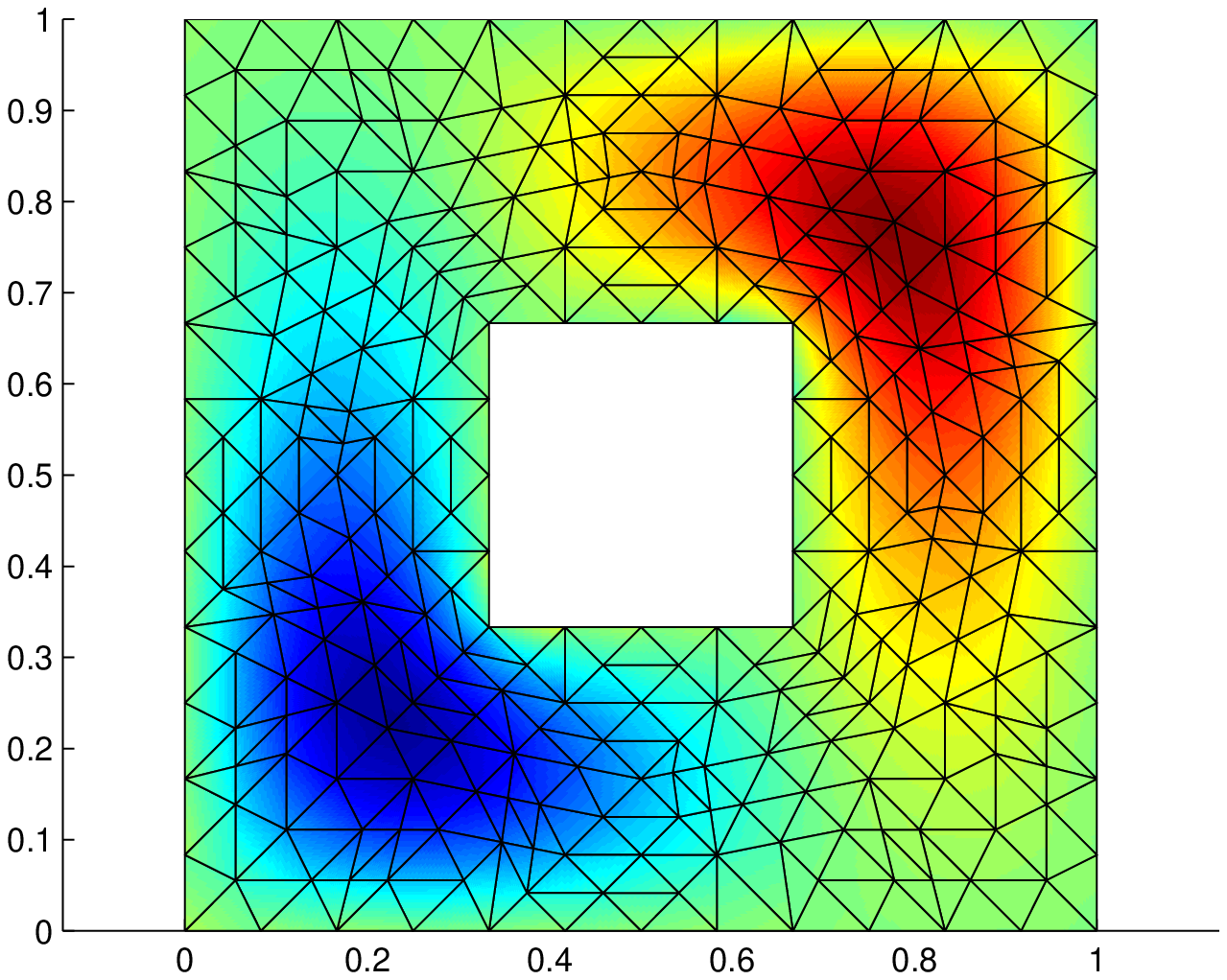}
\caption{Eigenfunctions corresponding to $\lambda_{2,h}$ and $\lambda_{3,h}$
after three levels of refinements: indicator based on $u_{2,h}$ and
$u_{3,h}$}
\label{fg:ring_irre_fun2}
\end{figure}

\begin{figure}
\includegraphics[width=6cm]{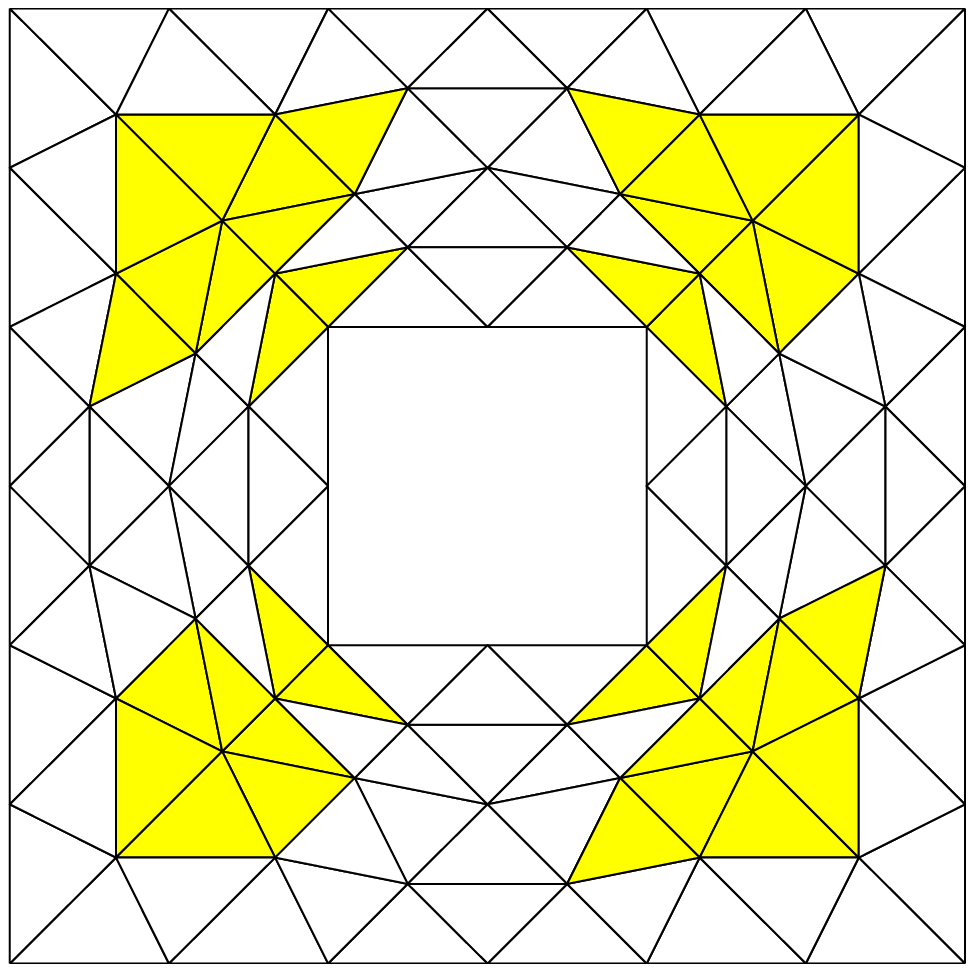}
\includegraphics[width=6cm]{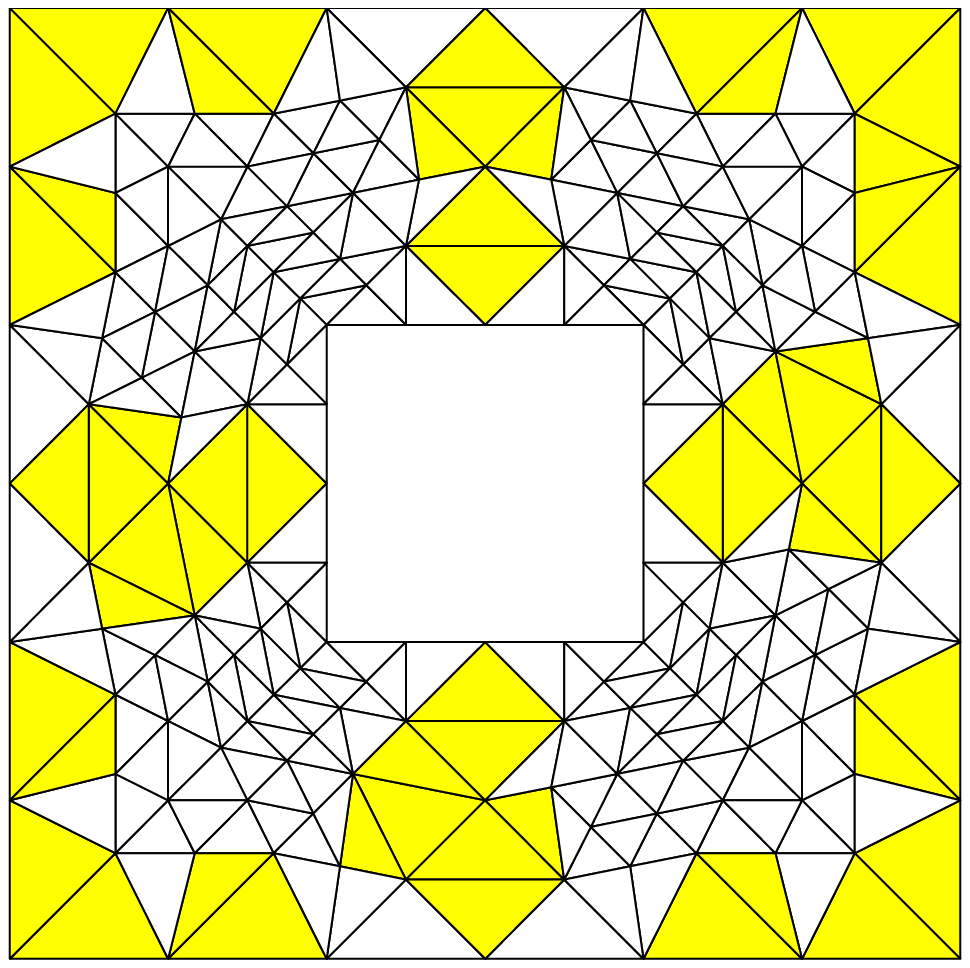}

\includegraphics[width=6cm]{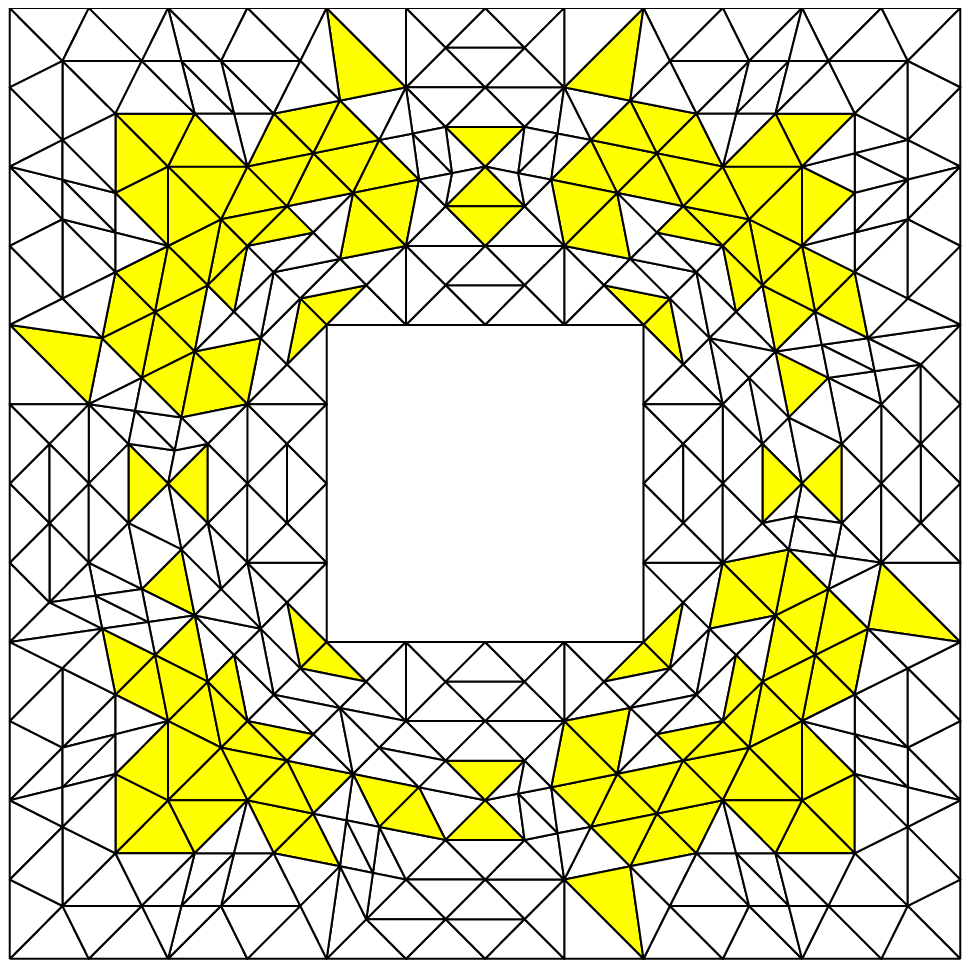}
\includegraphics[width=6cm]{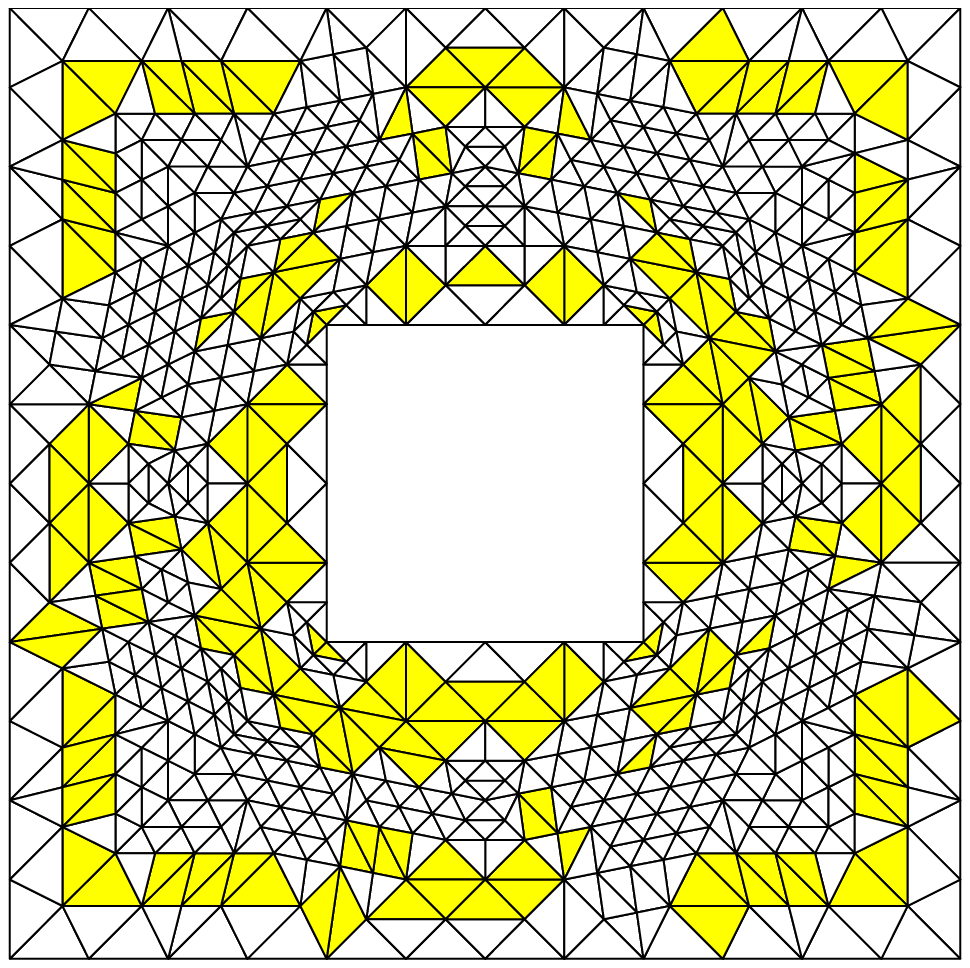}

\includegraphics[width=6cm]{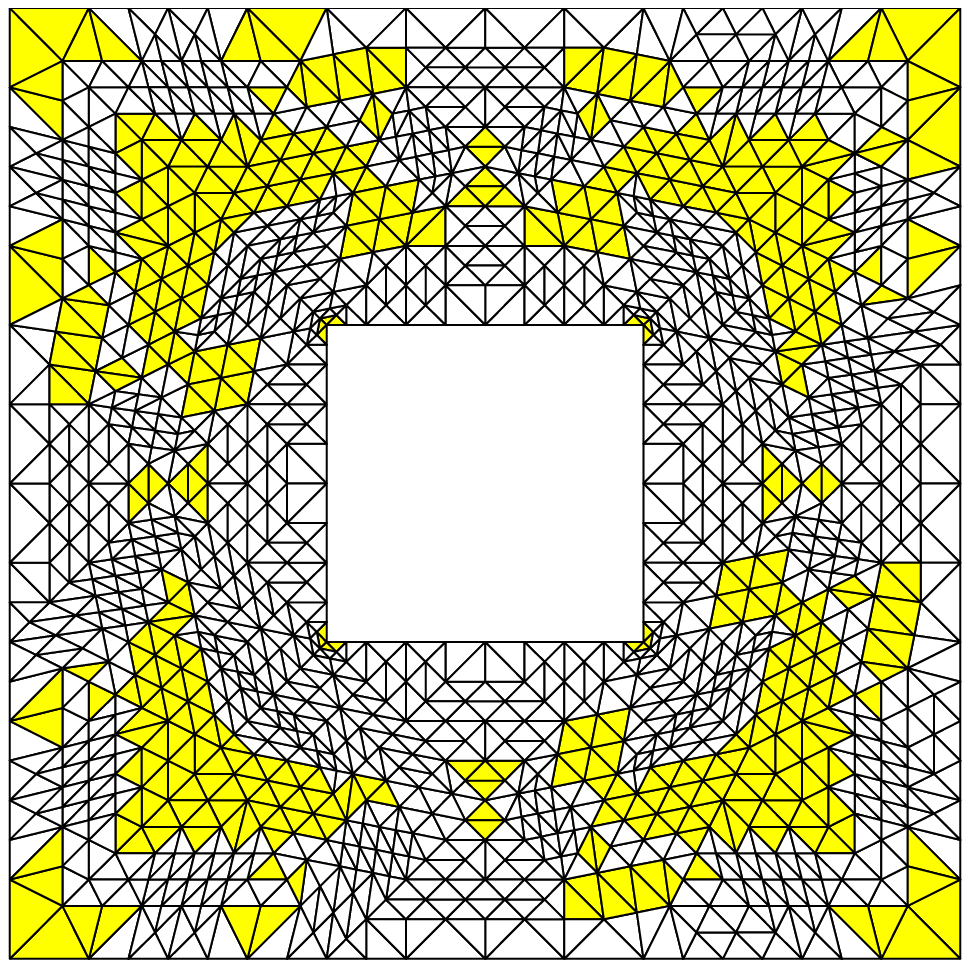}
\includegraphics[width=6cm]{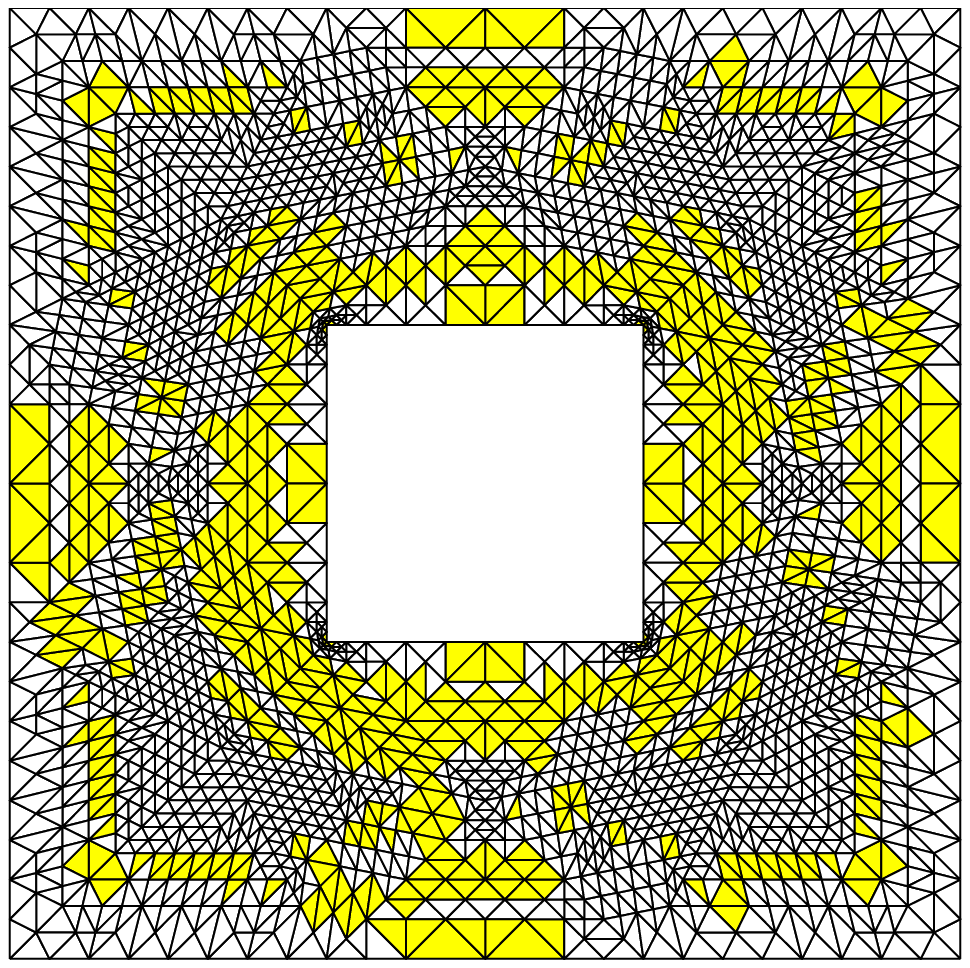}

\includegraphics[width=6cm]{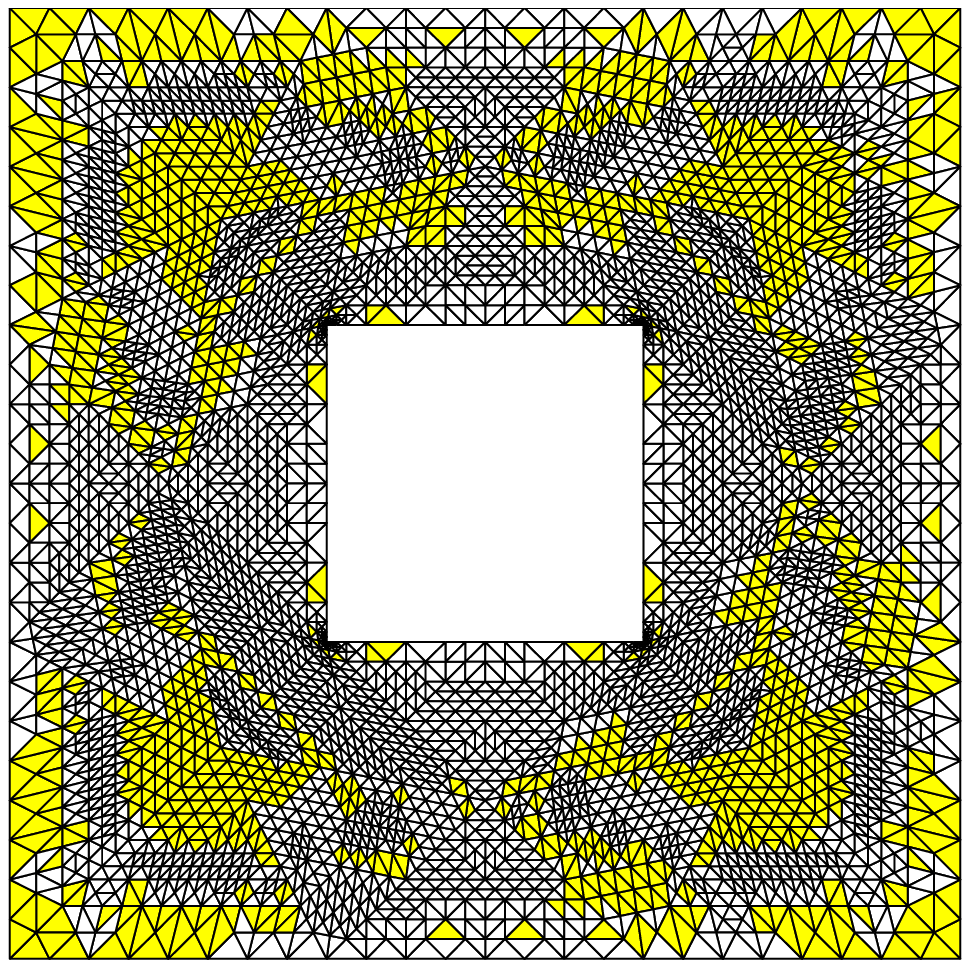}
\includegraphics[width=6cm]{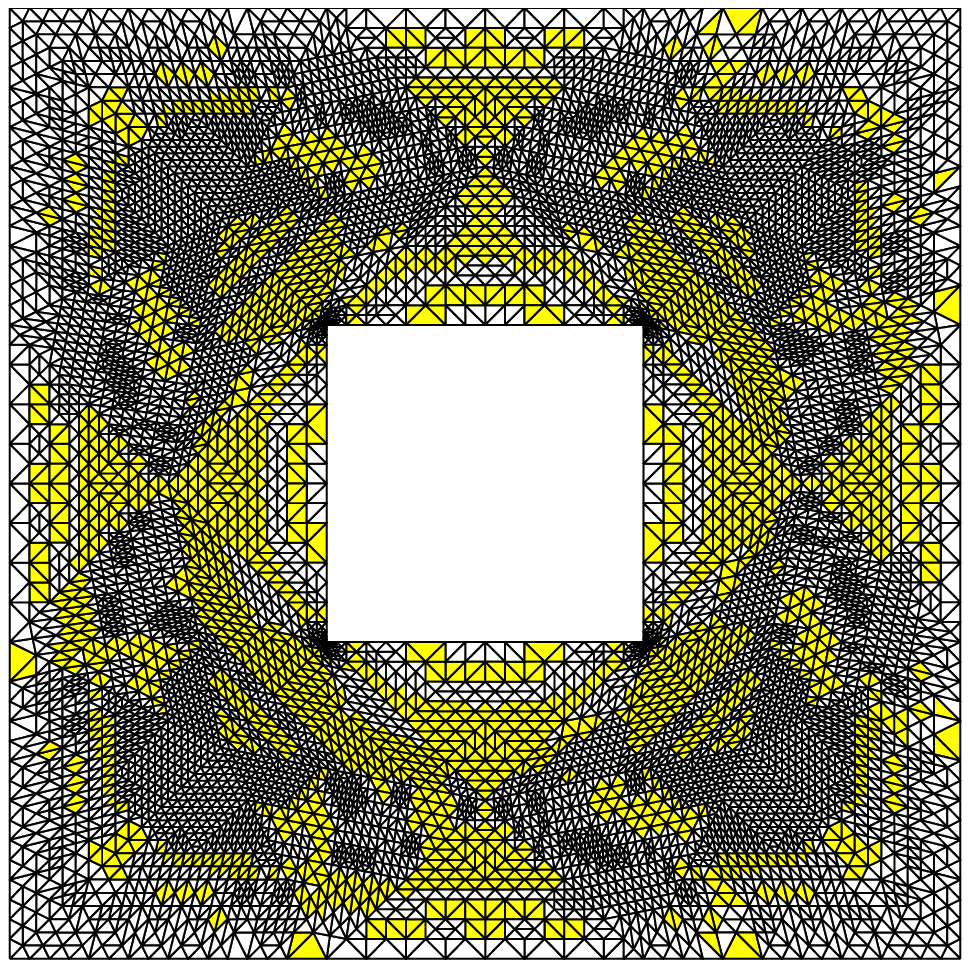}
\caption{Eight level of refinements of the square ring domain: indicator based
on $u_{2,h}$ and $u_{3,h}$}
\label{fg:ring-allmesh2}
\end{figure}

The convergence history of the adaptive algorithm is shown in
Figure~\ref{fg:ring_history2}. It can be seen that the procedure is performing
optimally with respect to the degrees of freedom.
\NEW{Finally, the effectivity index is reported in
Figure~\ref{fg:effectivity}; it can be seen that the ratio between the
error in the eigenvalues and the indicators is bounded above and below.}

\begin{figure}
\subfigure[$\lambda_{2,h}$]{
\includegraphics[width=6cm]{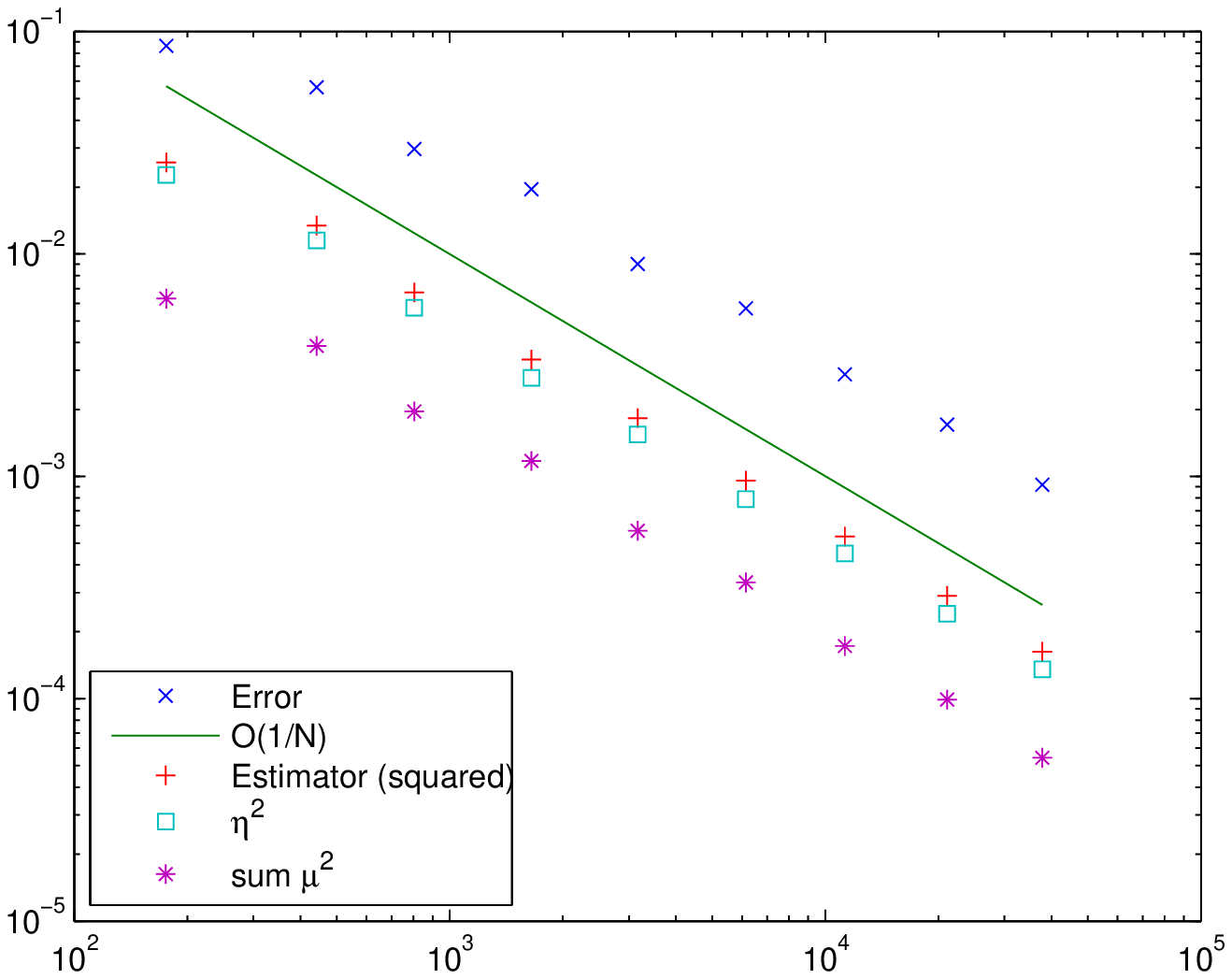}
}
\subfigure[$\lambda_{3,h}$]{
\includegraphics[width=6cm]{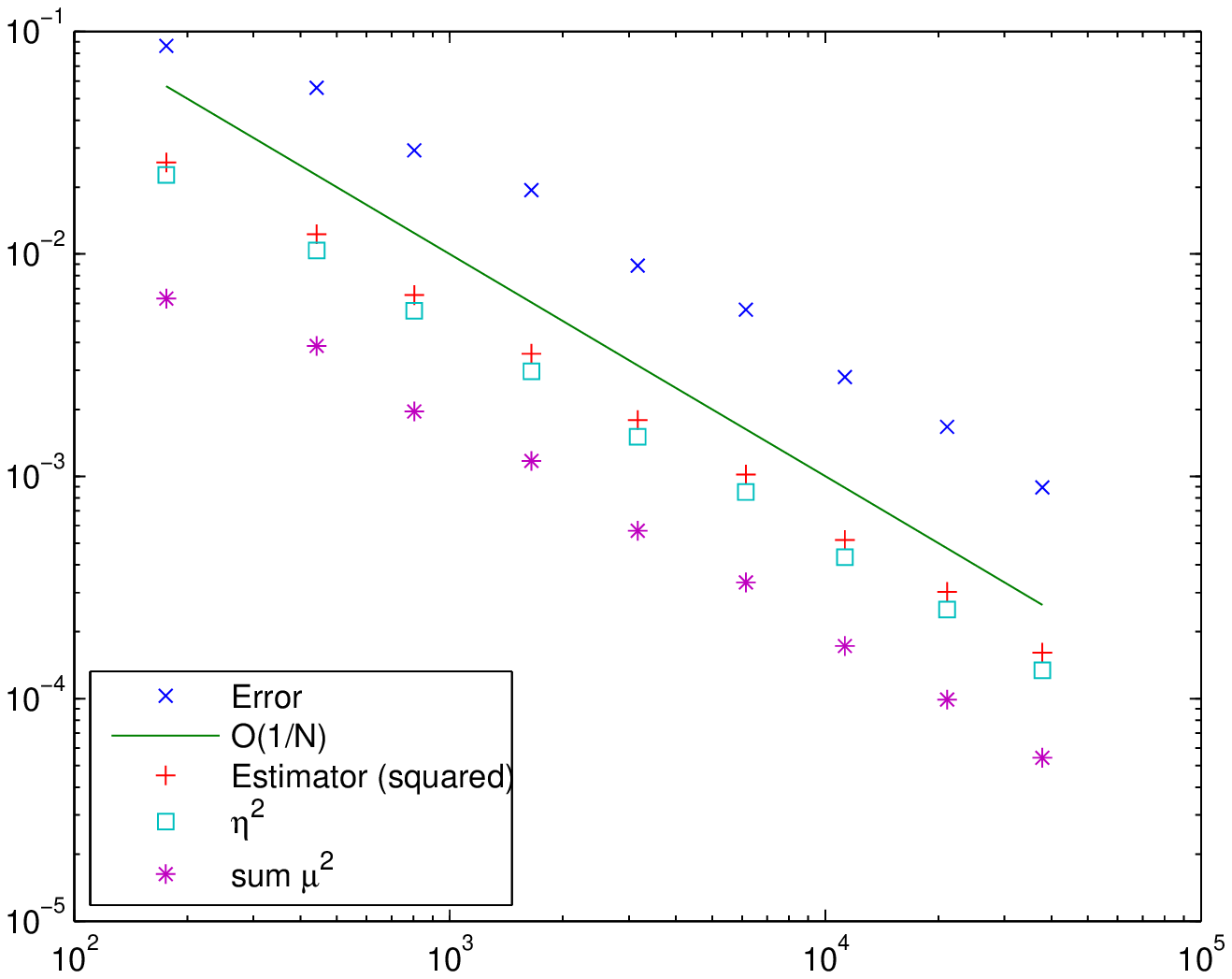}
}
\caption{Convergence history of the adaptive procedure: indicator based
on $u_{2,h}$ and $u_{3,h}$}
\label{fg:ring_history2}
\end{figure}

\begin{figure}
\includegraphics[width=6cm]{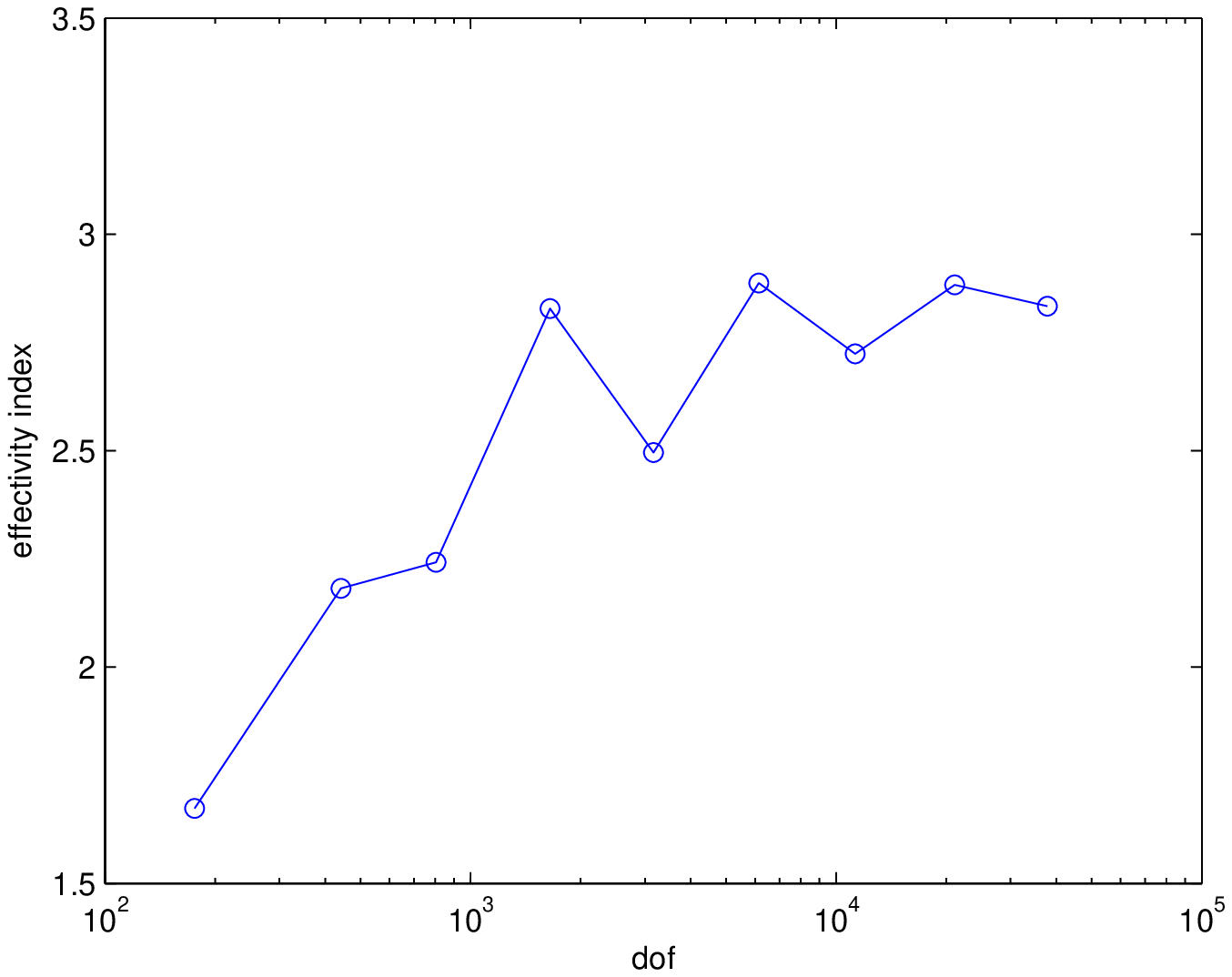}
\caption{Effectivity index: indicator based on $u_{2,h}$ and
$u_{3,h}$}
\label{fg:effectivity}
\end{figure}

\NEW{
\section*{Acknowledgements}
This work has started when R.G.~Dur\'an hosted D.~Boffi and L.~Gastaldi at the
Departamento de Matem\'atica of the University of Buenos Aires within the
Argentina-Italy bilateral projects {\it Innovative numerical
methods for industrial problems with complex and mobile
geometries} funded by CNR-CONICET (2011-2012) and by MinCyT-MAE
(2011-2013) IT/10/05 AR11M06.
}

\bibliographystyle{plain}
\bibliography{ref}

\end{document}